\documentclass[11pt]{amsart}

\usepackage[a4paper,hmargin=3.5cm,vmargin=4cm]{geometry}
\usepackage{amsfonts,amssymb,amscd,amstext}
\usepackage{graphicx}
\usepackage[dvips]{epsfig}

\usepackage{fancyhdr}
\usepackage{times}

\usepackage{enumerate}
\usepackage{titlesec}
\usepackage{mathrsfs}

\titleformat{\section}
{\filcenter\bfseries\large} {\thesection{.}}{0.2cm}{}
\titleformat{\subsection}[runin]
{\bfseries} {\thesubsection{.}}{0.15cm}{}[.]
\titleformat{\subsubsection}[runin]
{\em}{\thesubsubsection{.}}{0.15cm}{}[.]

\usepackage[up,bf]{caption}

\newtheorem{theorem}{Theorem}[section]

\newtheorem{lemma}[theorem]{Lemma}
\newtheorem{corollary}[theorem]{Corollary}

\theoremstyle{definition}
\newtheorem{definition}[theorem]{Definition}
\newtheorem{remark}[theorem]{Remark}

\newtheorem{problem}[theorem]{Problem}

\numberwithin{equation}{section}
\numberwithin{figure}{section}

\def\Dcal{\mathcal{D}}

\def\Pcal{\mathcal{P}}
\def\Rcal{\mathcal{R}}

\def\Ncal{\mathcal{N}}

\def\Tcal{\mathcal{T}}

\def\Ascr{\mathscr{A}}
\def\Cscr{\mathscr{C}}

\def\Uscr{\mathscr{U}}
\def\Bscr{\mathscr{B}}

\def\I{\mathtt{I}}

\def\c{\mathbb{C}}
\def\z{\mathbb{Z}}
\def\d{\mathbb{D}}
\def\b{\mathbb{B}}
\def\r{\mathbb{R}}
\def\n{\mathbb{N}}

\def\z{\mathbb{Z}}
\def\h{\mathbb{H}}
\def\jgot{\mathfrak{j}}
\def\igot{\mathfrak{i}}

\def\Agot{\mathfrak{A}}
\def\mgot{\mathfrak{m}}

\def\IA{\mathfrak{I}_{\mathfrak{A}}}

\def\dist{\mathrm{dist}}
\def\span{\mathrm{span}}
\def\length{\mathrm{length}}

\newcommand\wt{\widetilde}
\newcommand\wh{\widehat}

\newcommand\dibar{\overline\partial}

\usepackage{color}

\begin{document}

\fancyhead[LO]{The Calabi-Yau problem, null curves, and Bryant surfaces}
\fancyhead[RE]{A.\ Alarc\'on and F.\ Forstneri\v c}
\fancyhead[RO,LE]{\thepage}

\thispagestyle{empty}

\vspace*{1cm}
\begin{center}
{\bf\LARGE The Calabi-Yau problem, null curves, and \\ Bryant surfaces}

\vspace*{0.5cm}

{\large\bf Antonio Alarc\'on $\;$ and $\;$ Franc Forstneri\v c}
\end{center}


\vspace*{1cm}

\begin{quote}
{\small
\noindent {\bf Abstract}\hspace*{0.1cm} 
In this paper we prove that every bordered Riemann surface $M$ admits a complete proper null holomorphic embedding into a ball of the complex Euclidean 3-space $\c^3$. The real part of such an embedding is a complete conformal minimal immersion $M\to \r^3$ with bounded image. For any such $M$ we also construct proper null holomorphic embeddings $M\hookrightarrow \c^3$ with a bounded coordinate function;  these give rise to properly embedded null curves $M\hookrightarrow SL_2(\c)$ and to properly immersed Bryant surfaces $M\to\h^3$ in the hyperbolic 3-space. In particular, we give the first examples of proper Bryant surfaces with finite topology and of hyperbolic conformal type. The main novelty when compared to the existing results in the literature is that we work with a fixed conformal structure on $M$. This is accomplished by introducing a conceptually new method based on complex analytic techniques. One of our main tools is an approximate solution to Riemann-Hilbert boundary value problem for null curves in $\c^3$, developed in \S  \ref{sec:RH}. 
\vspace*{0.1cm}

\noindent{\bf Keywords}\hspace*{0.1cm} Riemann surfaces, complex curves, null holomorphic curves, minimal surfaces, Bryant surfaces,  complete immersions, proper immersions.

\vspace*{0.1cm}

\noindent{\bf MSC (2010):}\hspace*{0.1cm} 53C42; 32H02, 53A10, 32B15.}
\end{quote}


\section{Introduction}
\label{sec:intro}
The study of locally area minimizing surfaces, or {\em minimal surfaces},
has been a major topic of geometry ever since Euler observed in 1744 
that the only area minimizing surfaces of rotation in $\r^3$ are planes and catenoids.
After the grounbreaking work of Lagrange (1760) and Meusnier (1776) it became clear that 
a smooth immersed surface $M \to \r^3$ is minimal if and only if its mean curvature 
vanishes identically. Assuming that $M$ is orientable and fixing a conformal 
(equivalently, a complex) structure on $M$, it is natural to study immersions 
$M\to \r^3$ that are {\em conformal} (angle preserving); 
it is classical that such an immersion is minimal if and only if it is harmonic.
This observation naturally links the theory of minimal surfaces to complex analysis.
A simple computation (cf.\ Osserman's classical survey \cite{Osserman}) 
shows that a conformal minimal immersion $M\to \r^3$ is 
locally (on any simply connected subset of $M$) the real part 
$\Re F$ of a {\em holomorphic null curve} --- a holomorphic immersion
$F=(F_1,F_2,F_3)\colon M\to\c^3$ into the complex Euclidean 3-space which
is directed by the conical quadric subvariety
\begin{equation}
\label{eq:Agot}
\Agot=\{z=(z_1,z_2,z_3)\in\c^3\colon z_1^2+z_2^2+z_3^2=0\},
\end{equation}
in the sense that the derivative $F'=(F_1',F_2',F_3')$ with respect to any local holomorphic coordinate on $M$ takes values in $\Agot\setminus \{0\}$. Conversely, the real and the imaginary part of a holomorphic null curve $F\colon M\to \c^3$ are conformally immersed minimal surfaces in $\r^3$. Furthermore, the Riemannian metric $F^*ds^2$ on $M$ induced by $F$ is twice the one induced by the real part $\Re F\colon M\to\r^3$ or by the imaginary part $\Im F\colon M\to\r^3$; cf.\ \cite[p.\ 65]{Osserman}. Here $ds^2$ denotes the Euclidean metric on $\r^3$ or $\c^3$. An immersion $F:M\to\r^n$ is said to be {\em complete} if
the induced metric $F^*ds^2$ is a complete metric on $M$.

This connection enables the use of a rich array of complex analytic methods in 
the study of minimal surfaces. For example, applying Runge's approximation theorem on a planar labyrinth 
of compact sets, Nadirashvili \cite{Nad}  showed in 1996 that the disc $\d=\{z\in \c\colon |z|<1\}$ admits 
a {\em complete bounded minimal immersion}  $\d\to \r^3$, thereby disproving a famous conjecture of 
Calabi from 1965 \cite{Calabi}.  

When applying Nadirashvili's technique to more general Riemann surfaces, an issue 
is that Runge's theorem can not be used to control the placement in $\r^3$ of the entire surface, 
and one is forced to cut away small pieces of the original surface 
in order to keep the image suitably bounded. This of course changes the conformal structure on the 
surface, except in the simply connected case when the surface is conformally equivalent to the disc.

In this paper we introduce considerably stronger and more refined complex analytic tools to this subject: 
the {\em Riemann-Hilbert problem for null curves}, the gluing method for holomorphic sprays (which amounts to a nonlinear version of the $\dibar$-problem in complex analysis), and a technique of exposing boundary points of Riemann surfaces. Some of these methods were already employed in \cite{AF1} where we constructed proper complete holomorphic immersions of an arbitrary bordered Riemann surface to the unit ball of $\c^2$ and, more generally, to any Stein manifold of dimension $>1$. Somewhat surprisingly we discovered that the Riemann-Hilbert method can also be adapted to the construction of holomorphic null curves in $\c^3$; cf.\ Theorem \ref{th:RH} below. The proof of this result depends on the technique of gluing holomorphic sprays applied to the derivatives of null curves. This requires on the one hand to control the periods of some maps in the amalgamated spray in order to obtain well defined null curves by integration, and on the other hand rather delicate estimates to verify that the resulting null curves have the desired properties. 
%

We now give a sampling of what can be achieved with these methods.
Our first main result is a considerable extension of Nadirashvili's theorem from \cite{Nad}.

%
%
%
%
\begin{theorem}\label{th:intro}
Every bordered Riemann surface carries a complete proper null holomorphic embedding into 
the unit ball of $\c^3$.
\end{theorem}

By considering the real or the imaginary part of a complete bounded null curve $F\colon M\to\c^3$ 
and taking into account that $F^*ds^2=2(\Re F)^*ds^2= 2(\Im F)^*ds^2$ we obtain the following
corollary to Theorem \ref{th:intro}.

%
%
%
%
\begin{corollary}[On the conformal Calabi-Yau problem]
\label{co:r3}
Every bordered Riemann surface carries a conformal complete minimal immersion into 
$\r^3$ with bounded image. 
\end{corollary}

Theorem \ref{th:intro} is proved in \S \ref{sec:th}. In view of \cite[Theorem 2.4]{AF2} we can ensure that the resulting null curve is embedded. The main novelty in comparison to all existing results in the literature is that 

{\em we do not change the complex ($=$ conformal) structure on the source Riemann surface}.

We wish to emphasize that results concerning Riemann surfaces in which one controls the conformal structure are typically much more subtle, but also more interesting, than those without keeping track of the structure. For example, the problem of embedding open Riemann surfaces as closed complex curves in $\c^2$ is tractable if one only controls the topology of the surface \cite{AL3}, but is notoriously difficult and far from settled when trying to control the complex structure (see e.g.\ \cite{FW0,FW1} and the references therein.)

\begin{remark} 
Our proof of Theorem \ref{th:intro} does not provide any information on the asymptotic 
behavior of the real or the imaginary part of null curves. However, it is possible to adapt our methods,
in particular,  the Riemann-Hilbert problem (cf.\ Theorem \ref{th:RH} below), to the construction 
of conformal minimal immersions in $\r^3$. 
This allows us to prove that {\em every bordered Riemann surface admits a conformal complete proper 
minimal immersion into a ball of $\r^3$}, thereby improving Corollary \ref{co:r3}. 
We postpone the details to a future publication.
\end{remark}

Before proceeding to other topics and results, we pause for a moment to briefly recall the fascinating history of the Calabi-Yau problem for surfaces. 

In 1965, Calabi \cite{Calabi} conjectured the nonexistence of bounded complete minimal surfaces in $\r^3$.
Even more, he conjectured that there is no complete minimal surface in $\r^3$ with a bounded projection into a straight line. After groundbreaking counterexamples by Jorge and Xavier \cite{JX} (for the latter) and Nadirashvili \cite{Nad} (for the former), these conjectures were revisited by S.-T.\ Yau in his 2000 millenium lecture \cite{Yau2} where he posed several questions on the geometry of complete bounded minimal surfaces in $\r^3$. (See also Problem 91 in Yau's 1982 problem list \cite{Yau1}.) This is why the problem of constructing complete bounded minimal surfaces is known as {\em the Calabi-Yau problem for surfaces}.

Nadirashvili's technique from \cite{Nad} was the seed of several construction methods, providing a plethora of examples. In this line, Ferrer, Mart\'in, and Meeks \cite{FMM} found complete properly immersed minimal surfaces with arbitrary topology in any given either convex or smooth and bounded domain in $\r^3$. The case of finite topology had been settled earlier in \cite{AFM}. 
On the other hand, {\em embedded} complete minimal surfaces with finite genus and at most 
countably many ends are always proper in $\r^3$, hence unbounded 
(see Colding and Minicozzi \cite{CM} and Meeks, P\'erez, and Ros \cite{MPR}). 
The general problem for embedded surfaces remains open.

With such a rich collection of examples at hand, a natural question was whether there exists a complete bounded minimal surface in $\r^3$ whose conjugate surface is well defined and also bounded; equivalently, whether there exists a complete bounded null curve in $\c^3$. More precisely, Mart\'in, Umehara, and Yamada asked whether there exist complete properly immersed null curves in a ball and complete bounded embedded null curves in $\c^3$ \cite[Problems 1 and 2]{MUY1}.
%
%
This question was answered affirmatively by Alarc\'on and L\'opez \cite{AL2} who constructed complete properly immersed null curves with arbitrary topology in any given convex domain of $\c^3$. Their method relies on Runge's and Mergelyan's approximation theorems and is different from Nadirashvili's one. Very recently, Alarc\'on and Forstneri\v c \cite{AF2} proved that the general position of null curves in $\c^3$ is embedded; combining this with \cite{AL2} one obtains complete properly embedded null curves with the given topology (but without control of the conformal type) in any convex domain of $\c^3$ \cite[Corollary 6.2]{AF2}. 

Since complex submanifolds of a complex Euclidean space are minimal, the Calabi-Yau problem is closely related to an old question of Yang \cite{Yang1,Yang2} whether there exist complete bounded complex submanifolds of $\c^n$ for $n>1$. In the particular case of complex curves, Jones \cite{Jones} gave a simply connected immersed example in $\c^2$,  Mart\'in, Umehara, and Yamada \cite{MUY2} gave ones with arbitrary finite genus, and Alarc\'on and L\'opez \cite{AL2} provided examples with any given topological type. Only recently, a complete bounded {\em embedded} complex curve in $\c^2$ has been discovered \cite{AL-embedded}. As for higher dimensional examples, it was observed in \cite[Corollary 1]{AF1} that any bounded strongly pseudoconvex Stein domain admits a complete bounded holomorphic embedding into a high dimensional Euclidean space $\c^n$. Very recently, Globevnik \cite{Globevnik} settled Yang's question in an optimal way by proving the existence of complete bounded complex hypersurfaces in $\c^N$ for all $N>1$.

Concerning the second and more ambitious conjecture by Calabi, it was proved by Alarc\'on, Fern\'andez, and L\'opez \cite{Afer,AFL1,AFL2} that every nonconstant harmonic (resp., holomorphic) function $M\to \r$ (resp., $M\to\c$) on an open Riemann surface can be realized as a coordinate function of a conformal complete minimal immersion $M\to\r^3$ (resp., of a complete null holomorphic immersion $M\to\c^3$). This result can be viewed as an extension of the one by Jorge and Xavier \cite{JX}. It implies in particular 
that $M$ is the underlying Riemann surface of a complete minimal surface in $\r^3$ with a bounded projection into a straight line if and only if $M$ carries a nonconstant bounded harmonic function. 

The preceding discussion brings us to the  next main result of the paper.

%
%
%
%

\begin{theorem}
\label{th:boundedcoordinate}
Every bordered Riemann surface $M$ carries a proper holomorphic null embedding 
$F=(F_1,F_2,F_3)\colon M\to\c^3$ such that the function $F_3$ is bounded on $M$.
\end{theorem}

This means that $(F_1,F_2)\colon M\to \c^2$ is a proper immersion. Our construction ensures that $|F_3|$ can be chosen arbitrarily small on $M$, a fact which is also seen by rescaling. Since a translation maps null curves to null curves, this also gives examples with $c_1<|F_3|<c_2$ for any given pair of constants $0<c_1<c_2$. 

Comparing Theorem \ref{th:boundedcoordinate} to the existing results in the literature, we should recall that every open Riemann surface carries a conformal minimal immersion in $\r^3$ properly projecting into $\r^2\times\{0\}\subset\r^3$ \cite{AL1}, and a holomorphic null embedding in $\c^3$ properly projecting into $\c^2\times\{0\}\subset\c^3$ \cite{AF2}. The constructions in \cite{AL1} and \cite{AF2} use particular versions of Runge's theorem for minimal surfaces and null curves that do not enable one to have any control on the third coordinate function. The solutions to the Riemann-Hilbert problem, provided by Theorem \ref{th:RH} in this paper, allow us to overcome this problem in the case of null curves. To the best of the authors' knowledge, these are the first examples of proper null curves in $\c^3$ with a bounded coordinate function.

Theorem \ref{th:boundedcoordinate} is in contrast to the theorem of Hoffman and Meeks \cite{HM} asserting that the only properly immersed minimal surfaces in $\r^3$ contained in a halfspace are planes. (Similar phenomenona are known for proper holomorphic discs $\d\to\c^2$ \cite{FG2001}.) In particular, if $F\colon M\to\c^3$ is a null curve as those in Theorem \ref{th:boundedcoordinate}, then $\Re (e^{\imath t} F)\colon M\to\r^3$ is a complete {\em non-proper} conformal minimal immersion for all $t\in \r$. (Here $\imath=\sqrt{-1}$.)  On the other hand, null curves $(F_1,F_2,F_3)\colon M\to\c^3$ with the property that both $\Re(F_1,F_2)\colon M\to\r^2$ and $\Im(F_1,F_2)\colon M\to\r^2$ are proper maps were provided in \cite{AL-conjugate}.

Theorem \ref{th:boundedcoordinate} is proved in \S \ref{sec:boundedNC}; here is the main idea in the simplest case when $M$ is the disc $\d$. Choose a couple of orthogonal null vectors in $\c^2\times\{0\}\subset\c^3$, for example, $V_1=(1,\imath,0)$ and $V_2=(1,-\imath,0)$. We begin with the linear null embedding $\d\ni z\mapsto zV_1$. In the next step we use Theorem \ref{th:RH} (the Riemann-Hilbert problem for null holomorphic curves) to deform this embedding near the boundary into the direction $V_2$. Dropping the nullity condition, one could simply take $\d\ni z\mapsto zV_1+z^N V_2$ for a big integer $N>>0$. However, to get a null curve, the third component must be involved, but it can be chosen arbitrarily $\Cscr^0$ small if $N$ is big enough. Next we deform the map from the previous step again in the direction $V_1$, changing  the third component only slighly in the $\Cscr^0(\d)$ norm. Repeating this zig-zag procedure, the resulting sequence of holomorphic null maps converges to a desired proper null curve with a bounded third component. 

For a general bordered Riemann surface $M$ this construction is performed locally on small discs abutting the boundary $bM$, and the local modifications are assembled together by the technique of gluing holomorphic sprays of maps with control up to the boundary. This technique, which was first developed in \cite{DF2007,F2007}, uses bounded linear solution operators for the $\dibar$-equation on strongly pseudoconvex domains and the implicit function theorem on Banach spaces. In the case at hand we actually glue the derivative maps with values in the null quadric $\Agot\setminus\{0\}$ (\ref{eq:Agot}), ensuring that the periods over the loops in $M$ generating the homology group $H_1(M;\z)$ remain zero, so the new map integrates to a null curve. This gluing is enclosed in the proof of Theorem \ref{th:RH} in \S \ref{sec:RH}. Another ingredient used in the proof is a method of exposing boundary points of Riemann surface, first developed in \cite{FW0}, which we adapt to null curves.

%
%
%
%
Theorems \ref{th:intro} and \ref{th:boundedcoordinate} yield a nontrivial line of corollaries.
Recall that a null curve in the special linear group 
\[
SL_2(\c)= \left \{ z= \left(\begin{matrix} z_{11} & z_{12} 
	\cr z_{21} & z_{22} \end{matrix} \right) \in\c^4 
	\colon \det z= z_{11}z_{22}-z_{12}z_{21}=1\right \}
\] 
is a holomorphic immersion 
$F\colon M\to SL_2(\c)$ of an open Riemann surface $M$ into $SL_2(\c)$ which is directed by the quadric variety
\[
\mathfrak{E}=	\left \{ z= \left(\begin{matrix} z_{11} & z_{12} 
	\cr z_{21} & z_{22} \end{matrix} \right) \colon \det z= z_{11}z_{22}-z_{12}z_{21}=0\right \} 
	\subset \c^4.
\]
As before, to be directed by $\mathfrak{E}$ means that the derivative $F'\colon M\to\c^4$ with respect to any local holomorphic coordinate on $M$ belongs to $\mathfrak{E}\setminus \{0\}$.
The biholomorphic map $\Tcal \colon \c^3\setminus\{z_3=0\} \to SL_2(\c)\setminus\{z_{11}=0\}$, given by
\begin{equation}\label{eq:T}
	\Tcal(z_1,z_2,z_3) = \frac{1}{z_3}\left(\begin{matrix} 1 & z_1+\imath z_2 \cr 
	z_1-\imath z_2 & 	z_1^2+z_2^2+z_3^2
\end{matrix} \right),
\end{equation}
carries null curves into null curves (see Mart\'in, Umehara, and Yamada \cite{MUY1}). Furthermore, if $F=(F_1,F_2,F_3) \colon M\to \c^3$
is a proper null curve such that $c_1 <|F_3|<c_2$ on $M$ for a pair of constants $0<c_1<c_2$ (such null curves exist in view of Theorem \ref{th:boundedcoordinate}), then $G=\Tcal\circ F\colon M\to SL_2(\c)$ is
a proper null curve in $SL_2(\c)$. This proves the following.

\begin{corollary}
\label{cor:nullSL}
Every bordered Riemann surface admits a proper holomorphic null embedding into $SL_2(\c)$. Furthermore, there exist such embeddings with a bounded coordinate function.
\end{corollary}

The projection of a null curve in $SL_2(\c)$ to the real hyperbolic 3-space $\h^3=SL_2(\c)/SU(2)$ is a {\em Bryant surface}, i.e., a surface with constant mean curvature one in $\h^3$  \cite{Br}. More specifically, considering the hyperboloid model of the hyperbolic space 
\[
	\h^3=\{(x_0,x_1,x_2,x_3) \in \mathbb{L}^4\colon x_0^2=1+x_1^2+x_2^2+x_3^2,\ \, x_0>0\},
\]
where $\mathbb{L}^4$ denotes the Minkowski $4$-space with the canonical Lorentzian metric of signature $(-+ + +)$,
and the canonical identification 
\[
(x_0,x_1,x_2,x_3) \equiv  \left( \begin{matrix}
x_0+x_3 & x_1 +\imath x_2 \cr
x_1-\imath x_2 & x_0-x_3
\end{matrix} \right) \in SL_2(\c),
\]
it follows that $\h^3=\{A \cdot \overline{A}^T\colon A\in SL_2(\c)\},$ where $\bar\cdot$ and $\cdot^T$ mean complex conjugation and transpose matrix. In this setting, Bryant's projection, $SL_2(\c) \to \h^3$, $A\mapsto A\cdot \overline{A}^T$, maps null curves in $SL_2(\c)$ to {\em conformal} Bryant immersions (i.e., of constant mean curvature one) in $\h^3$.  Furthermore, proper null curves in $SL_2(\c)$ project to proper Bryant surfaces since the fibration $SL_2(\c)\to \h^3$ has a compact fiber $SU(2)$. Conversely, every simply connected Bryant surface in $\h^3$ lifts to a null curve in $SL_2(\c)$ \cite{Br}.  Bryant surfaces are special among constant mean curvature surfaces in $\h^3$ for many reasons. For instance, they are connected to minimal surfaces in $\r^3$ by the so-called {\em Lawson correspondence} \cite{La} which implies that every simply connected Bryant surface is isometric to a minimal surface in $\r^3$, and vice versa. This class of surfaces has been a fashion research topic that received many important contributions in the last decade; see e.g.\ \cite{UY,Ro,CHR} for background on Bryant surface theory.

By projecting to $\h^3$ proper null curves $M\to SL_2(\c)$, 
furnished by Corollary \ref{cor:nullSL}, we obtain the following result.

\begin{corollary}
\label{cor:nullH3}
Every bordered Riemann surface is conformally equivalent to a properly immersed Bryant surface in $\h^3$.
\end{corollary}

To the best of the authors' knowledge, these are the first examples of proper null curves in 
$SL_2(\c)$, and of Bryant surfaces in $\h^3$, with finite topology and {\em hyperbolic conformal type}, in the sense of carrying non-constant negative subharmonic functions. On the contrary, the geometry of regular ends of Bryant surfaces, which are conformally equivalent to a punctured disc, is very well understood. In particular, since properly embedded Bryant annular ends in $\h^3$ are regular (see Collin, Hauswirth, and Rosenberg \cite{CHR}), the examples in Corollary \ref{cor:nullH3} are not embedded.

The analogue of Corollary \ref{cor:nullH3} in the context of minimal surfaces in $\r^3$ was a longstanding open problem until 2003 when Morales \cite{Morales} gave the first example of a properly immersed minimal surface in $\r^3$ with finite topology and hyperbolic conformal type, thereby disproving an old conjecture of Sullivan (see Meeks \cite{Meeks}). More recently, it was shown that every open Riemann surface carries a proper conformal minimal immersion into $\r^3$ \cite{AL1} and a proper holomorphic null embedding into $\c^3$ \cite{AF2}. Thus, one is led to ask:

\begin{problem} Is every open Riemann surface conformally equivalent to a properly immersed Bryant surface in $\h^3$~? Even more, does every open Riemann surface admit a proper holomorphic null embedding (or immersion) into $SL_2(\c)$~?
\end{problem}

Obviously, not every open Riemann surface admits a null holomorphic immersion into $\c^3$ with a nonconstant bounded component function, as those in Theorem \ref{th:boundedcoordinate}; hence an affirmative answer to the above problem would require a different approach. On the other hand, this problem is purely complex analytic, in the sense that there are no topological obstructions. Indeed, combining our arguments in the proof of Theorem \ref{th:boundedcoordinate} with the Mergelyan theorem for null curves in $\c^3$ \cite{AF2,AL1} we obtain the following result.

\begin{theorem}
\label{th:giventopology}
Every orientable noncompact smooth real surface $N$ without boundary admits a complex structure $J$ such that the Riemann surface $(N, J)$ carries
\begin{enumerate}[1]
\item[$\bullet$] a proper holomorphic null embedding $(F_1,F_2,F_3)\colon (N, J)\to\c^3$ such that the function $F_3$ is bounded on $N$; hence
\item[$\bullet$] a proper holomorphic null embedding $(N, J)\to SL_2(\c)$, and
\item[$\bullet$] a proper conformal Bryant immersion $(N, J)\to \h^3$.
\end{enumerate}
\end{theorem}

Theorem \ref{th:intro} also has interesting ramifications concerning null curves in $SL_2(\c)$ and Bryant surfaces in $\h^3$. Note that the transformation $\Tcal$ \eqref{eq:T} maps complete bounded null curves in $\c^3 \cap \{|z_3|>c\}$ for any $c>0$ into complete bounded null curves in $SL_2(\c)$ \cite{MUY1}. Furthermore, the Riemannian metric on an open Riemann surface $M$, induced by a null holomorphic immersion $M\to SL_2(\c)$, is twice the one induced by its Bryant's projection $M\to\h^3$; cf.\ \cite{Br,UY}. Hence Theorem \ref{th:intro} implies the following corollary. 

\begin{corollary}\label{co:SLbounded}
Every bordered Riemann surface admits a complete null holomorphic embedding into 
$SL_2(\c)$ with bounded image, and it is conformally equivalent to a complete bounded immersed Bryant surface in $\h^3$.
\end{corollary}

Complete bounded immersed null holomorphic discs in $SL_2(\c)$, hence complete bounded simply-connected Bryant surfaces, were provided in \cite{FMUY,MUY1}. Existence results in the line of Corollary \ref{co:SLbounded}, for surfaces with arbitrary topology but without control on the conformal structure on the Riemann surface, can be found in  \cite{AL2,AF2}.

We expect that our solutions to the Riemann-Hilbert problem for null curves (Theorem \ref{th:RH}) will be useful for further developments in this classical subject.


\section{Preliminaries}\label{sec:prelim}

We denote by $\langle\cdot,\cdot\rangle$, $|\cdot|$, $\dist(\cdot,\cdot)$, and $\length(\cdot)$, respectively, the hermitian inner product, norm, distance, and length on $\c^n$, $n\in\n$. Hence $\Re \langle\cdotp,\cdotp\rangle$ is the Euclidean inner product on $\r^{2n}\cong\c^n$. 
Given $u\in\c^n$, we set $\langle u\rangle^\bot=\{v\in\c^n\colon \langle u,v\rangle=0\}$ and $\span\{u\}=\{\zeta u\colon \zeta\in\c\}$. Obviously, if $u,v\in\c^n$, $u\neq 0$, then $v+\langle u\rangle^\bot$ is the complex affine hyperplane in $\c^n$ passing through $v$ and orthogonal to $u$. 

If $K$ is a compact topological space and $f\colon K\to\c^n$ is a continuous map, we denote by $\|f\|_{0,K}=\sup_{x\in K} |f(x)|$ the maximum norm of $f$ on $K$. If $M$ is a smooth manifold, $K\subset M$, and $f$ is of class $\Cscr^1$, we denote by $\|f\|_{1,K}$ the $\Cscr^1$-maximum norm of $f$ on $K$, measured with respect the expression of $f$ in a system of local coordinates in some fixed finite open cover of $K$. Similarly we define these norms for maps $M\to X$ to a smooth manifold $X$ by using a fixed cover by coordinate patches on both manifolds.


\subsection{Bordered Riemann surfaces}
\label{subsec:Riemann}

If $M$ is a topological surface with boundary, we denote by $bM$ the $1$-dimensional topological manifold determined by the boundary points of $M$. We say that a surface is {\em open} if it is non-compact and does not contain any boundary points. 

A {\em Riemann surface} is an oriented surface together with the choice of a complex structure.

An open connected Riemann surface, $M$, is said to be a {\em bordered Riemann surface} if $M$ is the interior of a compact one dimensional complex manifold, $\overline{M}$, with smooth boundary $bM\ne \emptyset$ consisting of fini\-tely many closed Jordan curves. The closure $\overline M=M\cup bM$ of such $M$ is called a {\em compact bordered Riemann surface}.

A domain $D$ in a bordered Riemann surface $M$ is said to be a {\em bordered domain} if it has smooth boundary; such $D$ is itself a bordered Riemann surface with the complex structure induced from $M$. It is classical that every bordered Riemann surface is biholomorphic to a relatively compact bordered domain in a larger Riemann surface. We denote by $\Bscr(M)$ the family of relatively compact bordered domains $\Rcal\Subset M$ such that $M$ is a tubular neighborhood of $\overline{\Rcal}$; that is to say, $\overline{\Rcal}$ is {\em holomorphically convex} (also called {\em Runge}) in $M$, and $M \setminus  \overline{\Rcal}$ is the union of finitely many paiwise disjoint open annuli. 

Let $M$ be a bordered Riemann surface and $X$ be a complex manifold. Given a number $r\ge 0$, we denote by $\Ascr^r(M,X)$ the set of all maps $f\colon \overline M \to X$ of class $\Cscr^r$ that are holomorphic on $M$. For $r=0$ we write $\Ascr^0(M,X)=\Ascr(M,X)$. When $X=\c$, we write $\Ascr^r(M,\c) = \Ascr^r(M)$ and $\Ascr^0(M) = \Ascr(M)$. Note that $\Ascr^r(M,\c^n) = \Ascr^r(M)^n$ is a complex Banach space. For any complex manifold $X$ and any number $r\ge 0$, the space $\Ascr^r(M,X)$ carries a natural structure of a complex Banach manifold \cite{F2007}. 

Let $M$ be a bordered Riemann surface of genus $g$ and $m$ ends; i.e., $\overline M$ has $m$ boundary components. The 1-st homology group $H_1(M;\z)$ is then a free abelian group on $l=2g+m-1$ generators which are represented by closed, smoothly embedded loops $\gamma_1,\ldots,\gamma_l\colon S^1 \to M$ that only meet at a chosen base point $p\in M$. Let $\Gamma_j=\gamma_j(S^1)\subset M$ denote the trace of $\gamma_j$. Their union $\Gamma=\bigcup_{j=1}^l \Gamma_j$ is a wedge of $l$ circles meeting at $p$.


\subsection{Null curves in $\c^3$}
\label{subsec:null}

Denote by $\Agot$ the quadric subvariety \eqref{eq:Agot} of $\c^3$. Vectors in $\Agot\setminus\{0\}$ are said to be {\em null}. Let $\imath=\sqrt{-1}$. Note that 
\[
	\Agot\setminus\{0\}=\{(1-\xi^2,\imath(1+\xi^2),2\xi)\zeta\colon \xi\in\c,\, \zeta\in\c\setminus\{0\}\};
\]
hence $\Agot$ is a complex conical submanifold of $\c^3$ contained in no finite union of real or complex hyperplanes. Every complex linear hyperplane of $\c^3$ intersects $\Agot\setminus\{0\}$.

\begin{definition} 
\label{def:null}
Let $M$ be an open Riemann surface. A holomorphic immersion $F\colon M\to\c^3$ is said to be {\em null} if its derivative $F'$ with respect to any local holomorphic coordinate on $M$ assumes values in $\Agot\setminus\{0\}$. The same definition applies if $M$ is a compact bordered Riemann surface with smooth boundary $bM \subset M$ and $F$ is of class $\Cscr^1(M)$ and holomorphic in the interior $\mathring M= M\setminus bM$.
\end{definition}

Given a bordered Riemann surface $M$, we denote by $\IA(M)$ the set of all $\Cscr^1$ null immersions $\overline{M}\to\c^3$ that are holomorphic in $M$.  If $F=(F_1,F_2,F_3)\in \IA(M)$ then $(dF_1)^2+(dF_2)^2+(dF_3)^2$ vanishes identically on $M$; here $d$ denotes the complex differential. Conversely, if $\Phi=(\phi_1,\phi_2,\phi_3)$ is an exact vectorial holomorphic $1$-form and $\phi_1^2+\phi_2^2+\phi_3^2$ vanishes everywhere on $M$, then $\Phi$ integrates to a null holomorphic immersion
\[
		F\colon M\to\c^3,\quad F(x)=\int^x \Phi,\quad x\in M.
\] 
The triple $\Phi=dF$ is said to be the {\em Weierstrass representation} of $F$. 

We denote by $\sigma_F^2$ the conformal Riemannian metric in $M$ induced by the Euclidean metric of $\c^3$ via a holomorphic immersion $F\colon M\to \c^3$; i.e.,
\[
		\sigma_F^2(x;v) = \langle dF_x(v), dF_x(v)\rangle=|dF_x(v)|^2, \quad x\in M,\ v\in T_x M.
\]
We denote by $\dist_{(M,F)}(\cdot,\cdot)$ the distance function in the Riemannian surface $(M,\sigma_F^2)$.

An arc $\gamma\colon [0,1)\to M$ is said to be {\em divergent} if $\gamma$ is a proper map; that is, if the point $\gamma(t)$ leaves any compact subset of $M$ when $t\to 1$.

\begin{definition}
\label{def:complete_imm}
Let $M$ be an open Riemann surface. A null holomorphic immersion $F\colon M\to\c^3$ is said to be {\em complete} if $(M,\sigma_F^2)$ is a complete Riemannian surface; that is, if the image curve $F\circ \gamma$ in $\c^3$ has infinite Euclidean length for any divergent arc $\gamma$ in $M$. 
\end{definition}

If $F\colon M\to\c^3$ is a null holomorphic immersion, then its real and imaginary parts $\Re F, \Im F\colon M\to\r^3$ are conformal minimal immersions (cf.\ Osserman \cite{Osserman}); furthermore, $\Re F$ and $\Im F$ are isometric to each other, and the Riemannian metric in $M$ induced by the Euclidean metric of $\r^3$ via them equals $\frac{1}{2}\, \sigma_F^2$ \cite[p.\ 65]{Osserman}. In particular, if one of the immersions $F$, $\Re F$, and $\Im F$ is complete, then so are the other two.


\section{Riemann-Hilbert problem for null curves in $\c^3$}
\label{sec:RH}

In this section we construct approximate solutions to Riemann-Hilbert boundary value problems for holomorphic null curves in $\c^3$. Theorem \ref{th:RH} below is the main new analytic tool for the construction of null curves developed in this paper. With this result at hand, we will prove Theorem \ref{th:intro} by adapting the method developed in \cite{AF1}. The same result is the key ingredient in the proof of Theorem \ref{th:boundedcoordinate}. We expect that this technique will find many subsequent applications.

We proceed in three steps. First, Lemma \ref{lem:RH} gives an approximate solution to the Riemann-Hilbert problem when the central curve is a null disc. In \S \ref{ss:families} we consider families of null discs and obtain an estimate that is used in the general case. Finally, Theorem \ref{th:RH} in \S\ref{ss:general} pertains to the case of an arbitrary bordered Riemann surface as the central null curve. 


\subsection{Riemann-Hilbert problem for null discs}
\label{ss:disc}
Recall that $\d=\{z\in\c:|z|<1\}$ is the unit disc and $\Agot$ is the null quadric \eqref{eq:Agot}. 

%
%
%
\begin{lemma} 
\label{lem:RH}
Let $F\colon\overline{\d}\to\c^3$ be a null holomorphic immersion (cf.\ Def.\ \ref{def:null}), let $\vartheta\in\Agot\setminus\{0\}$ be a null vector (the {\em direction vector}), let $\mu\colon b{\d}=\{\zeta\in\c\colon |\zeta|=1\}\to \r_+ := [0,+\infty)$ be a continuous function (the {\em size function}), and set
\[
\varkappa\colon b{\d}\times\overline{\d}\to\c^3,\quad \varkappa(\zeta,\xi)=F(\zeta)+\mu(\zeta)\,\xi\,\vartheta.
\]
Given numbers $\epsilon>0$ and $0<r<1$, there exist a number $r'\in [r,1)$ and a null holomorphic immersion $G\colon\overline{\d}\to\c^3$ satisfying the following properties:
\begin{enumerate}[\it i)]
\item $\dist(G(\zeta),\varkappa(\zeta,b{\d}))<\epsilon$ for all $\zeta\in b\d$, 
\item $\dist(G(\rho\zeta),\varkappa(\zeta,\overline{\d}))<\epsilon$ for all $\zeta\in b\d$ and all $\rho\in [r',1)$, and 
\item $G$ is $\epsilon$-close to $F$ in the $\Cscr^1$ topology on $\{\zeta\in\c\colon |\zeta|\leq r'\}$.
\end{enumerate}
Furthermore, if $I$ is a compact arc in $b\d$, the size function $\mu$ vanishes everywhere on $b\d\setminus  I$, and $U$ is an open neighborhood of $I$ in $\overline{\d}$, then 
\begin{enumerate}[\it i)]
\item[\it iv)] one can choose $G$ to be $\epsilon$-close to $F$ in the $\Cscr^1$ topology on $\overline{\d}\setminus  U$.
\end{enumerate}
\end{lemma}

\begin{remark}
\label{rem:reparat}
Lemma \ref{lem:RH} holds on any smoothly bounded simply connected domain $D\Subset \c$. A conformal map $D\to \d$, furnished by the Riemann mapping theorem, extends to a smooth diffeomorphism $\phi\colon \overline D\to\overline \d$; by using such $\phi$ we can transport the Riemann-Hilbert data from $D$ to $\d$ and then transport a solution on $\d$ back to a solution on $D$. Choose an annular neighborhood $A\subset \overline D$ of the boundary curve $C=bD$ and a retraction $\rho\colon A\to C$. 
A natural statement of the lemma, which is clearly equivalent to the above statement when $D=\d$, is then the following:

\noindent (*) Given a number $\epsilon>0$ and a compact set $K\subset D$, there exist a null curve $G\in \IA(D)$ and a compact set $K'\subset D$ such that $K\subset K'$, $K'\cup A=\overline D$, and the following hold:
\begin{enumerate}[\it i')]
\item $\dist(G(\zeta),\varkappa(\zeta,b{\d}))<\epsilon$ for all $\zeta\in bD$,
\item $\dist \bigl(G(\rho(\zeta),\varkappa(\rho(\zeta),\overline{\d})\bigr) < \epsilon$ for all $\zeta\in \overline D\setminus K'$, and 
\item $G$ is $\epsilon$-close to $F$ in the $\Cscr^1$ topology on $K'$. 
\end{enumerate}
\end{remark}

\begin{proof}[Proof of Lemma \ref{lem:RH}]
We begin by constructing a null holomorphic immersion $G\colon\overline \d\to\c^3$ enjoying properties {\it i)}, {\it ii)}, and {\it iii)}; condition {\it iv)} will be considered in the second part of the proof. Replacing $\mu$ by $\mu |\vartheta|$ and $\vartheta$ by $\vartheta/|\vartheta|$ we may assume that $|\vartheta|=1$. 

Consider the unbranched two-sheeted holomorphic covering
\[
\pi\colon\c^2\setminus\{(0,0)\}\to \Agot\setminus\{0\},\quad \pi(u,v)=\big(u^2-v^2,\imath(u^2+v^2),2uv\big).
\]
Since $\overline{\d}$ is simply connected, the derivative $F'\colon\overline{\d}\to\Agot\setminus\{0\}$ admits a holomorphic $\pi$-lifting $(u,v)\colon\overline{\d}\to\c^2\setminus\{(0,0)\}$ such that
\begin{equation}\label{eq:F'}
		F'=\pi(u,v)=\big(u^2-v^2,\imath(u^2+v^2),2uv\big).
\end{equation}
Likewise, there exists a vector $(p,q)\in\c^2\setminus\{(0,0)\}$ such that
\begin{equation}\label{eq:vartheta}
		\vartheta=\pi(p,q)=\big(p^2-q^2,\imath(p^2+q^2),2pq\big) \in \Agot\setminus\{0\}.
\end{equation}
By compactness of $\overline{\d}$ we can fix a number $C_0>0$ satisfying \begin{equation}\label{eq:boundM}
		\|(pu-qv,pu+qv, qu+pv)\|_{0,\overline{\d}}< C_0.
\end{equation}

Set $\eta=\sqrt{\mu}\colon b{\d}\to \r_+$. We approximate $\eta$ uniformly on $b{\d}$ 
by a Laurent polynomial  
\begin{equation}
\label{eq:eta}
		\tilde\eta(\zeta)=\sum_{j=1}^N A_j \zeta^{j-m}
\end{equation}
with a pole of order $m-1$ at the origin, where the $A_j$'s are complex numbers and $N$ and $m$ natural ones. The function $\tilde\eta^2$ then approximates $\mu$ uniformly on $b{\d}$ and is of the form
\begin{equation}\label{eq:B}
			\tilde\eta^2(\zeta)=\sum_{j=1}^{2N} B_j \, \zeta^{j-2m};\quad B_1,\ldots,B_{2N}\in\c.
\end{equation}
Observe that for any integer $k\ge m$ the function $\zeta\mapsto \zeta^k \tilde\eta(\zeta)$ is a holomorphic polynomial fixing the origin  $0\in\c$. If the approximation of $\eta$ by $\tilde\eta$ is close enough on $b\d$, there exists a number $r'\in [r,1)$ such that
\begin{equation}\label{eq:r'}
		|F(\rho\zeta)-F(\zeta)| + |\tilde\eta^2(\rho\zeta)-\mu(\zeta)|<\epsilon/2 \quad 
		\forall \zeta\in b{\d},\; \rho\in[r',1).
\end{equation}

For any integer $n\ge m$ we define functions $u_n,\, v_n \colon \overline\d \to\c$ of class $\Ascr(\d)$ by
\begin{equation}\label{eq:tildef}
	u_n(\xi) =  u(\xi)+\sqrt{2n+1} \,\tilde\eta(\xi) \, \xi^n p, \quad 
	v_n(\xi) =  v(\xi)+\sqrt{2n+1} \,\tilde\eta(\xi) \, \xi^n q.
\end{equation}
Consider the map  of class $\Ascr(\d)$ given by 
\begin{equation}\label{eq:Phi}
		\Phi_n = \pi(u_n,v_n) =
		\big(u_n^2-v_n^2,\imath (u_n^2+v_n^2),2u_nv_n \big)
	  \colon \overline{\d}\to \Agot\subset \c^3.
\end{equation}
After an arbitrarily small change of the vector $(p,q)\in \c^2$ we may assume that the map $(u_n, v_n) \colon\overline{\d}\to\c^2$ does not assume the value $(0,0)$ for any sufficiently large $n\in \n$. Indeed, it suffices to choose $(p,q)$ such that $(u(\xi),v(\xi))+s(p,q) \ne (0,0)$ for all $\xi\in b\d$ and $s\in\c$, a condition which holds by dimension reasons for almost all vectors $(p,q)\in\c^2$. By continuity the same is then true for all $\xi\in \overline \d$ satisfying $r\le |\xi|\le 1$ for some $r<1$ sufficiently close to $1$; hence $(u_n(\xi),	v_n(\xi)) \ne (0,0)$ for such $\xi$ and for any $n\in \n$. Since $\sqrt{2n+1}\, |\xi^n| \to 0$ uniformly on $\{|\xi|\le r\}$ as $n\to +\infty$, we see that $(u_n,v_n)$ converges to $(u,v)$ uniformly on $|\xi|\le r$ as $n\to +\infty$, and hence the range of the map $(u_n,v_n)$ avoids the value $(0,0)$ for all sufficiently big $n$. For such $(p,q)\in\c^2\setminus\{(0,0)\}$ and $n\in \n$, the map $\Phi_n=\pi(u_n,v_n)$ (\ref{eq:Phi}) has range in $\Agot\setminus\{0\}$, and hence it furnishes a null holomorphic immersion 
$G_n \in\IA(\d)$ by the expression
\[
	G_n(\zeta)=F(0)+\int_0^\zeta \Phi_n(\xi)\, d\xi,  \qquad \zeta\in \overline\d.
\]
Taking into account \eqref{eq:F'} and \eqref{eq:vartheta}, a straightforward computation gives
\begin{equation}\label{eq:G}
		G_n(\zeta)=F(\zeta)+ {\bf B}_n(\zeta)\, \vartheta + {\bf A}_n(\zeta),
\end{equation}
where
\[
	{\bf B}_n(\zeta) = (2n+1)\sum_{j=1}^{2N} \int_0^\zeta B_j \xi^{2n+j-2m}\,d\xi 
	=\sum_{j=1}^{2N}  \frac{2n+1}{2n+1+j-2m}   {B_j\, \zeta^{2n+1+j-2m}}\in\c, 
\] 
\[
	{\bf A}_n(\zeta) = 2\sqrt{2n+1} \int_0^\zeta \sum_{j=1}^N  
	A_j\, \xi^{n+j-m} \big(u(\xi)(p,\imath p,q)+v(\xi)(-q,\imath q,p) \big)\, d\xi\in\c^3.
\]

Since the coefficients $(2n+1)/(2n+1+j-2m)$ in the sum for ${\bf B}_n$ converge to $1$ as $n\to +\infty$, there exists an $n_0\in \n$ such that for all $n\ge n_0$ we have
\begin{equation}\label{eq:estimateB}
	\sup_{|\zeta|\le 1} \big| {\bf B}_n(\zeta) - \zeta^{2n+1}\tilde\eta^2(\zeta)\big| < \epsilon/4.
\end{equation}
(See also the expression \eqref{eq:B} for $\tilde\eta^2(\zeta)$.) 
The remainder term ${\bf A}_n(\zeta)$ in \eqref{eq:G} can be estimated as follows; the first inequality uses \eqref{eq:boundM} and the last one $|\zeta|\leq 1$:
\begin{equation}\label{eq:estimateA}
	|{\bf A}_n(\zeta)| \leq 2\sqrt{2n+1} \, C_0 \sum_{j=1}^N  |A_j| 
	\int_0^{|\zeta|} |\xi|^{n+j-m}\, d|\xi| 
		\le 2 C_0 \sum_{j=1}^N  \frac{\sqrt{2n+1}}{n+1+j-m}|A_j|.
\end{equation}
It follows that $|{\bf A}_n|\to 0$ uniformly on $\overline{\d}$ as $n\to +\infty$. In view of \eqref{eq:G} and \eqref{eq:estimateB}, we obtain for big enough $n\in\n$ that
\begin{equation}\label{eq:Gsolves}
	 \sup_{|\zeta|\le 1} \big|G_n(\zeta)- \big(F(\zeta) + \tilde\eta^2(\zeta)\zeta^{2n+1}\vartheta\big)\big|<\epsilon/2. 
\end{equation}

Let us verify that the null disc $G=G_n:\overline\d\to\c^3$ satisfies conditions {\it i)}, {\it ii)}, 
and {\it iii)} in the lemma provided that $n$ is chosen sufficiently big. 
For $\zeta=e^{\imath s}\in b{\d}$ one has the estimate
\begin{eqnarray*}
	\dist(G(\zeta),\varkappa(\zeta,b{\d})) & \leq & 
	\big|G(e^{\imath s})-\varkappa(e^{\imath s},e^{\imath s (2n+1)})\big| \\
  & \stackrel{\eqref{eq:Gsolves}}{<} & 
  \big|\tilde\eta^2(e^{\imath s})e^{\imath s(2n+1)}\vartheta - 
  \mu(e^{\imath s})e^{\imath s(2n+1)}\vartheta \big| + \frac{\epsilon}{2} <\epsilon;
\end{eqnarray*}
for the last inequality use \eqref{eq:r'} and recall that $|\vartheta|=1$. This ensures property {\it i)}.
Likewise, for $\zeta=e^{\imath s}\in b{\d}$ and $\rho\in[r',1)$, we infer that
\begin{eqnarray*}
	\dist(G(\rho\zeta),\varkappa(\zeta,\overline{\d})) & \leq & 
	\big| G(\rho e^{\imath s})-\varkappa(e^{\imath s},\rho^{2n+1} e^{\imath s(2n+1)}) \big| \\
 & \stackrel{\eqref{eq:Gsolves}}{<} & \frac{\epsilon}2+|F(\rho e^{\imath s})-F(e^{\imath s})|\\
 & & + \big| \tilde\eta^2(\rho e^{\imath s})\rho^{2n+1}e^{\imath s(2n+1)}
 		\vartheta - \mu(e^{\imath s})\rho^{2n+1}e^{\imath s (2n+1)}\vartheta) \big|.
\end{eqnarray*}
Since $\rho\in[r',1)$ and $|\vartheta|=1$, the above is not greater than
\[
	\big| F(\rho e^{\imath s})-F(e^{\imath s})\big| +
	\big|\tilde\eta^2(\rho e^{\imath s})-\mu(e^{\imath s})\big|
	+ \frac{\epsilon}{2} \stackrel{\eqref{eq:r'}}{<}\epsilon,
\]
thereby proving property {\it ii)}. Finally, it is clear from \eqref{eq:F'} and \eqref{eq:Phi} that 
$\Phi_n \to F'$ uniformly on the set $\{\zeta\in\c\colon |\zeta|\leq r'\}$ as $n\to+\infty$; hence
property {\it iii)} is achieved provided that $n$ is chosen big enough. This completes the proof for a general size function $\mu$.

\medskip
\noindent{\em Proof of property {\it iv)}}. 
Assume now that $I$ is a compact arc in the circle $b{\d}$ and that the function $\mu\colon b\d\to\r_+$ vanishes on $b{\d}\setminus I$. Pick an open neighborhood $U$ of $I$ in $\overline{\d}$. Let $\Dcal\Subset\c$ be a smoothly bounded convex domain such that $\overline{\d}\subset\overline{\Dcal}$, $\overline{\d}\setminus U \subset \Dcal$, and  $I':=b{\Dcal}\cap b{\d}$ is a compact arc containing $I$ in its interior (see Fig.\ \ref{fig:RHiv}). Such $\Dcal$ is obtained by pushing the boundary curve of $\d$ slightly out of $\overline \d$, except on a small neighborhood of the arc $I$ where it remains fixed.
\begin{figure}[ht]
    \begin{center}
    \resizebox{0.38\textwidth}{!}{\includegraphics{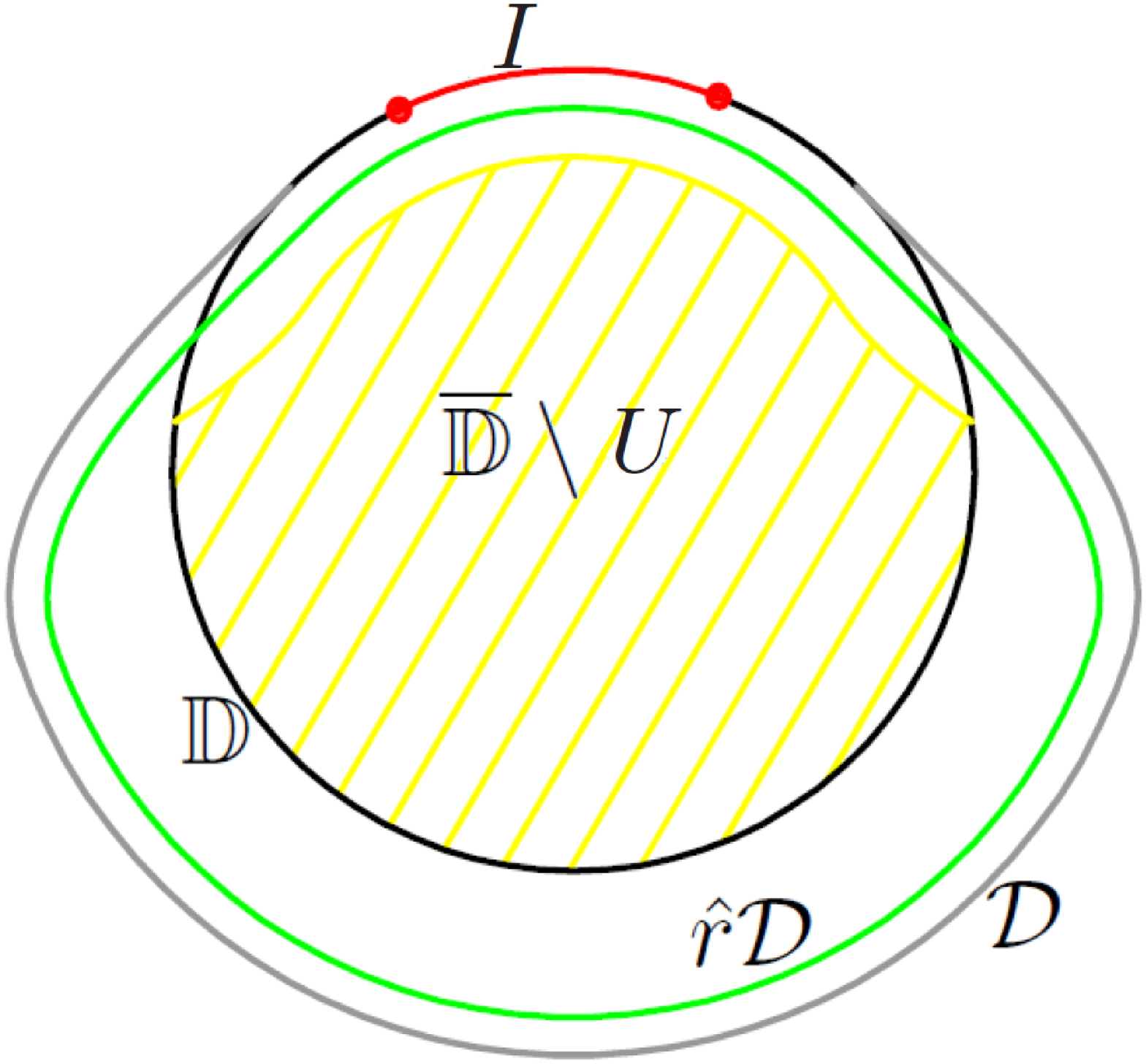}}
        \end{center}
        \vspace{-0.25cm}
\caption{$\Dcal$.}\label{fig:RHiv}
\end{figure}  
By the Mergelyan theorem for null curves (cf.\ \cite{AL1}, \cite[Lemma 1]{AL2}, and \cite{AF2}) we can approximate the null disc $F\colon \overline \d\to\c^3$ by a null holomorphic immersion $\wh F\colon \overline \Dcal\to\c^3$. The proof is completely elementary over a simply connected domain: we write the derivative of $F\colon \overline \d \to\c^3$ in the form $F'=\pi(u,v)$ \eqref{eq:Phi}, with $(u,v)\colon \overline \d\to\c^2\setminus \{(0,0)\}$. By the usual Mergelyan theorem and the general position argument we can approximate $(u,v)$ in the $\Cscr(\overline \d)$ topology by a map $(\hat u,\hat v)\colon  \overline \Dcal \to\c^2\setminus\{(0,0)\}$ of class $\Ascr(\Dcal)$ and take $\pi(\hat u,\hat v)$ as the derivative of the new null immersion $\wh F\colon \overline \Dcal\to\c^3$. Consider the continuous function $\wh\mu\colon b {\Dcal}\to \r_+$ given by 
\[
		\wh\mu|_{b {\Dcal}\cap b{\d}}=\mu|_{b {\Dcal}\cap b{\d}} \quad\text{and}\quad \wh\mu|_{b {\Dcal}\setminus b{\d}}=0.
\]
Define 
\[
\wh \varkappa\colon b {\Dcal}\times\overline{\d}\to\c^3,\quad \wh \varkappa(\zeta,\xi)=F(\zeta)+\wh\mu(\zeta)\xi\vartheta.
\]
Pick a number $\hat r\in[r,1)$ close enough to $1$ such that 
\[
		\overline{\d}\setminus U \subset \hat r \Dcal:=\{\hat r \zeta\colon \zeta\in\Dcal\}\Subset\Dcal
\]
and
\[
		|\wh F(\rho \zeta)- \wh F(\zeta)|<\epsilon/2\quad 
		\text{for all $\zeta\in \overline{\Dcal}$ and all $\rho\in [\hat r,1)$.}
\]

Since $\Dcal$ is conformally equivalent to the disc $\d$ by a diffeomorphism extending smoothly to the closures, the already proved case of the lemma can be applied to $\wh F$ on $\overline \Dcal$ to find a number $r'\in[\hat r,1)$ and a null holomorphic immersion $\wh G\in\IA(\Dcal)$ satisfying the following properties (which correspond to {\it i')}--{\it iii')} in Remark \ref{rem:reparat}):
\begin{enumerate}[\rm (a)]
\item $\dist(\wh G(\zeta), \wh \varkappa(\zeta,b{\d}))<\epsilon/2$ for all $\zeta\in b {\Dcal}$.
\item $\dist(\wh G(\rho\zeta), \wh \varkappa(\zeta,\overline{\d}))<\epsilon/2$ for all $\zeta\in b {\Dcal}$ and all $\rho\in[r',1)$.
\item $\wh G$ is $\frac{\epsilon}2$-close to $F$ in the $\Cscr^1$ topology on $r'\overline{\Dcal}\subset\Dcal$.
\end{enumerate}

It is now straightforward to verify that the null immersion $G:=\wh G|_{\overline{\d}}\in\IA(\d)$ meets the requirements {\it i)}--{\it iv)}. In particular, since the compact set $\overline{\d}\setminus U$ is contained in $\Dcal$, and hence in $r'\Dcal$ for some $r'<1$, property {\it iv)} of $G$ follows directly from property {\rm (c)} of $\wh G$. The same observation gives {\it iii)}, and we leave the trivial verification of {\it i)}--{\it ii)} to the reader.
\end{proof}


\subsection{Riemann-Hilbert problem for families of null discs}
\label{ss:families}

An inspection of the proof of Lemma \ref{lem:RH} gives the analogous conclusion for a family $F_t\colon \overline\d \to \c^3$ of null discs of class $\Ascr^1(\d)$ depending holomorphically on a parameter $t\in B\subset \c^N$ in a ball $B$ centered at the origin in a complex Euclidean space, provided that the $\pi$-lifting $(u_t,v_t)\colon\overline\d\to \c^2$ of the derivatives $F'_t=\pi(u_t,v_t) \colon \overline\d\to \Agot\setminus \{0\}$ (cf.\ \eqref{eq:F'}) satisfy the following condition:
\begin{equation}
\label{eq:generalposition}
	(u_t(\xi),v_t(\xi))+s(p,q)\ne (0,0) \quad \forall \xi\in b\d,\ s\in \c,\ t\in B. 
\end{equation}

\begin{lemma}
\label{lem:RHfamily}
Assume that $\{F_t\}_{t\in B}$ is as above. For every sufficiently big $n\in\n$ there exists a holomorphic family of null discs $G_{t,n}\colon \overline\d\to \c^3$ $(t\in B)$ of class $\Ascr^1(\d)$ and of the form
\begin{equation}\label{eq:Gt}
		G_{t,n}(\zeta) = F_t(\zeta)+{\bf B}_n(\zeta)\, \vartheta + 
		{\bf A}_{t,n}(\zeta),\quad \zeta\in \overline\d,\ t\in B,
\end{equation}
where ${\bf B}_n$ is independent of $t$ and $|{\bf A}_{t,n}|\to 0$ as $n\to +\infty$, uniformly on $B'\times \overline \d$ for any smaller ball $B'\Subset B$. In particular, if $n$ is big enough then  for every $t\in B'$ the map $G_{t,n}$ satisfies the conclusion of Lemma \ref{lem:RH} with respect to $F_t$ and a given $\epsilon>0$.

Furthermore, there exist a constant $\delta>0$ and an integer $n_0\in \n$ with the following property. Write $g_n(t,\zeta)=G'_{t,n}(\zeta)$ (the derivative with respect to $\zeta$). If $n\ge  n_0$ and $\beta \colon B'\times \overline \d \to B$ is a map of class $\Ascr(B'\times \d)$ and of the form $\beta(t,\zeta)=t+b(t,\zeta)$, with $\|b\|_{0,\overline{B'}\times \d} <\delta$, then the null disc
\begin{equation}\label{eq:wtGt}
	\wt G_{t,n}(\zeta) = F_t(0) + \int_0^\zeta g_n\bigl(t+b(t,\xi),\xi\bigr)\, 
							d\xi,\quad \zeta\in\overline\d,
\end{equation}
satisfies the conclusion of Lemma \ref{lem:RH} with respect to $F_t$ for every $t\in B'$.
\end{lemma}

The case of Lemma \ref{lem:RHfamily} with a variable parameter depending on the base point will be used 
in the proof of Theorem \ref{th:RH} below.

\begin{proof} The first part is seen by a straightforward inspection of the proof of Lemma \ref{lem:RH}. In particular, the estimate  \eqref{eq:estimateA} holds uniformly in $t\in B'\Subset B$.

We now consider the second part.
Write $f(t,\cdotp):=F'_t=\pi(u_t,v_t)$ for $t\in B$; see \eqref{eq:F'}. Denote by ${\bf A}_{t,n}={\bf A}_n(t,\cdotp)$ and ${\bf B}_n$ the corresponding quantities \eqref{eq:G}. By choosing $n$ big enough, we may assume that $G_{t,n}$ satisfies Lemma \ref{lem:RH} with respect to $F_t$, with $\epsilon$ replaced by $\epsilon/2$. Note that the term ${\bf B}_n$ is independent of $t$, while $t$ enters in ${\bf A}_n(t,\cdotp)$ only through $(u_t,v_t)$. The equation \eqref{eq:G}, written for the $\zeta$-derivatives, is
\begin{equation}\label{eq:gt}
    g_n(t,\zeta) = f(t,\zeta) + {\bf B}'_n(\zeta)\, \vartheta + 
		{\bf A}'_n(t,\zeta), \quad \zeta\in \overline\d, \ t\in B.
\end{equation}
Replacing $t$ by $\beta(t,\zeta)=t+b(t,\zeta)$ gives 
\begin{equation}\label{eq:gtat}
	  g_n\bigl(t+b(t,\zeta),\zeta\bigr) = f\bigl(t+b(t,\zeta),\zeta\bigr) 
	  			+ {\bf B}'_n(\zeta)\, \vartheta + {\bf A}'_n\bigl(t+b(t,\zeta),\zeta\bigr)
\end{equation}
for all $\zeta\in \overline\d$ and $t\in B'$. The map $\wt G_{t,n}$ \eqref{eq:wtGt} is obtained by integrating this expression. 

To conclude the proof, it suffices to show that for any given constant $\sigma>0$ we can choose the number $\delta>0$ (which controls the size of the function $b$ in \eqref{eq:wtGt}) such that 
\begin{equation} \label{eq:Gtminus}
		\|G_{t,n}-\wt G_{t,n}\|_{0,\overline \d} <\sigma\quad  \forall t\in B'.
\end{equation}
Properties {\it i)} and {\it ii)} in Lemma \ref{lem:RH} then follow by choosing $\sigma=\epsilon/2$, and the property {\it iii)} follows from \eqref{eq:Gtminus} by Cauchy estimates provided that $\sigma>0$ is small enough. 

To prove the estimate \eqref{eq:Gtminus}, we compare the integrals of the terms on the right hand side of \eqref{eq:gtat} to the integrals of the corresponding terms in \eqref{eq:gt}. Since the second term does not contain the variable $t$, we need not consider it. For the first term we have by continuity
\[
		\big|f(t,\zeta) - f\bigl(t+b(t,\zeta),\zeta\bigr)\big| < \epsilon/4  
		\quad\forall t\in B',\ \zeta\in\overline\d,
\]
provided that $\delta>0$ is small enough, so the integrals of these two functions along a straight line segment from $0$ to any point $\zeta\in \overline \d$ differ by at most $\epsilon/4$. 

It remains to estimate the last term in \eqref{eq:gtat}. Recall that $\pi(p,q)=\vartheta$ is the direction null vector. By compactness there is a constant $C_0>0$ satisfying 
\begin{equation}\label{eq:boundMt}
		\|(pu_t-qv_t,pu_t+qv_t, qu_t+pv_t)\|_{0,\overline{\d}}< C_0\quad \forall t\in B'.
\end{equation}
Using this constant, the estimate \eqref{eq:estimateA} for ${\bf A}_n$ remains valid and shows that the integral of the last term in \eqref{eq:gtat} converges to zero uniformly on $B'\times \overline \d$ as $n\to +\infty$. 
\end{proof}


\subsection{Riemann-Hilbert problem for bordered Riemann surfaces as null curves}
\label{ss:general}
We are now ready to prove the following main technical result of this paper.

\begin{theorem}\label{th:RH}
Let $M$ be a bordered Riemann surface and set $\overline M=M\cup bM$. Let $I_1,\ldots,I_k$ be a finite collection of pairwise disjoint compact subarcs of $bM$ which are not connected components of $bM$. Choose a small annular neighborhood $A\subset \overline M$ of $bM$ (that is, $A$ consists of pairwise disjoint annuli, each of them containing a connected component of $bM$ in its boundary) and a smooth retraction $\rho\colon A\to bM$. Let $F\colon \overline M \to\c^3$ be a null holomorphic immersion of class $\IA(M)$ (the {\em central curve}). Let $\vartheta_1,\ldots,\vartheta_k \in \Agot\setminus\{0\}$ be null vectors (the {\em direction vectors}), let $\mu \colon bM \to \r_+$ be a continuous function supported on $\bigcup_{i=1}^k I_i$ (the {\em size function}), and consider the continuous map  
\[
	\varkappa\colon bM \times\overline{\d}\to\c^3,\qquad 
	\varkappa(x,\xi)=\left\{\begin{array}{ll}
	F(x); & x\in bM\setminus \bigcup_{i=1}^k I_i \\
	F(x) + \mu(x)\,\xi \,\vartheta_i; & x\in I_i,\; i\in\{1,\ldots,k\}.
	\end{array}\right.
\]
Then for any number $\epsilon>0$ there exist an arbitrarily small open neighborhood $\Omega$ of\ \ $\bigcup_{i=1}^k I_i$ in $A$ 
and a null holomorphic immersion $G\colon \overline M\to\c^3$ of class $\IA(M)$ satisfying the following properties:
\begin{enumerate}[\it i)]
\item $\dist(G(x),\varkappa(x,b\d))<\epsilon$ for all $x\in bM$.
\item $\dist(G(x),\varkappa(\rho(x),\overline{\d}))<\epsilon$ for all $x\in \Omega$.
\item $G$ is $\epsilon$-close to $F$ in the $\Cscr^1$ topology on $\overline M\setminus \Omega$.
\end{enumerate}
\end{theorem}

At this time we are unable to prove Theorem \ref{th:RH} if the support of the size function $\mu$ contains
a complete boundary curve of $M$, except if $M$ is the disc when it coincides with Lemma \ref{lem:RH}.
However, the stated result suffices for all our applications in this paper. For results without 
the nullity condition see \cite{Cerne2004,FG2001} and the references therein.

\begin{proof}
For the sake of simplicity we assume that $k=1$; the general case trivially follows from a finite recursive application of this particular one. Write $I=I_1$ and $\vartheta=\vartheta_1$. We denote by $C$ the connected component of $bM$ containing the arc $I$.  

Choose  oriented  closed curves $\Gamma_1,\ldots, \Gamma_l\subset M$ whose homology classes are a basis of the 1st homology group $H_1(M;\z)$ (see \S \ref{sec:prelim}). Let $\gamma_j\colon [0,1]\to \Gamma_j$ be a parametrization of $\Gamma_j$ and set $\Gamma=\bigcup_{j=1}^l \Gamma_j$. Pick a nowhere vanishing holomorphic 1-form $\theta$ on $M$ (such exists by the Oka-Grauert principle; cf.\ Theorem 5.3.1 in 
\cite[p.\ 190]{F2011}). Then $dF=f\theta$ where $f=(f_1,f_2,f_3)\colon \overline M\to \Agot \setminus\{0\}$ is a map of class $\Ascr(M)$ with values in the null quadric. 

We denote by 
\begin{equation}
\label{eq:P}
	\Pcal=(\Pcal_1,\ldots, \Pcal_l)\colon \Ascr(M,\c^3)\to (\c^3)^l
\end{equation}
the {\em period map} whose $j$-th component, applied to $f\in \Ascr(M,\c^3)$, equals 
\[
	\Pcal_j(f) = \int_{\gamma_j} f\theta = \int_0^1 f(\gamma_j(s))\, \theta(\gamma_j(s),\dot{\gamma_j}(s))\, ds
	\in \c^3.
\]

By \cite[Theorem 2.3]{AF2} we can approximate the null immersion $F$ arbitrarily  close in the $\Ascr^1(M,\c^3)$ topology by a {\em regular null immersion}, i.e., by one for which the tangent planes to $\Agot\setminus\{0\}$ at all points $f(x)$ for  $x \in M$ together span $\c^3$ (cf.\ Definition 2.2 in \cite{AF2}; here $f=dF/\theta$). Assume from now on that $F$ is regular. By \cite[Lemma 5.3]{AF2} there exist an open ball $B$ centered at the origin in some Euclidean space $\c^N$ and a holomorphic map 
\[
	B \times \overline M \ni (t,x) \longmapsto f_t(x) \in \Agot \setminus\{0\}, 
\]
with $f_0=f$, such that the period map $B\ni t \mapsto \Pcal(f_t(\cdot)) \in (\c^3)^l$ (\ref{eq:P}) has maximal rank $3l$ at $t=0$. By increasing $N$ we can also ensure that for every point $x\in \overline M$ the map
\[
	\frac{\partial}{\partial t}\bigg|_{t=0} f_t(x) \colon T_0\c^N= \c^N\longrightarrow T_{f(x)} (\Agot\setminus\{0\})
\]
is surjective. This means that the family of maps $f_t \colon \overline M\to \Agot\setminus\{0\}$, $t\in B$, is a (local) {\em dominating holomorphic spray of class $\Ascr(M)$} with values in $\Agot\setminus\{0\}$ in the sense of \cite[Definition 4.1]{DF2007}. (See also \cite[\S 5.9]{F2011}.) The domination condition can be obtained by applying to an existing (possibly non-dominating) spray $f_t(x)$ the local flows (with independent complex time variables) of holomorphic vector fields which generate the tangent space of the null quadric at every point of $\Agot\setminus\{0\}$. For instance, we can use the linear vector fields 
\[
		V_0=z_1\frac{\partial}{\partial z_1} + z_2\frac{\partial}{\partial z_2} 
		+ z_3\frac{\partial}{\partial z_3} \quad\text{and}\quad
		V_{i,j}=z_i\frac{\partial}{\partial z_j} - z_j\frac{\partial}{\partial z_i},\quad 1\le i\ne j\le 3.
\]

Choose a smoothly bounded, connected and simply connected domain $D\subset M$ such that $\overline D\cap \Gamma =\emptyset$, $I':=bD\cap bM$ is a proper subarc of the curve $C$, and the arc $I$ is contained in the relative interior of $I'$. Since the function $\mu\colon C\to \r_+$ is supported in $I$, we can extend it to a continuous function $\mu\colon bD\to \r_+$ that vanishes on $bD\setminus I$. 

The maps $f_t \colon \overline M\to \Agot\setminus\{0\}$ do not necessarily integrate to null curves since their periods over the curves $\Gamma_j$ need not vanish for $t\ne 0$. However, on $\overline D$ there is no period problem, so $f_t\,\theta$ restricted to $\overline D$ is the differential of a null immersion 
\begin{equation}
\label{eq:Ft}
		F_t(x)=F(x_0)+\int_{x_0}^x  f_t\,\theta,\qquad x\in \overline D,\ t\in B.
\end{equation}
Here $x_0\in D$ is an arbitrary initial point and the integral is independent of the choice of a path in $\overline D$ from $x_0$ to $x$. Note that for $t=0$ the map $F_0$ coincides with the restriction to $\overline D$ of our given null immersion $F\colon \overline M\to \c^3$.

Choose a conformal diffeomorphism $h\colon\overline D\to\overline\d$ with $h(x_0)=0$. Denote by $\zeta$ the variable on $\d$. Since $\theta$ and $dh$ are nonvanishing holomorphic 1-forms on $\overline D$, their quotient $\alpha=\theta/dh \colon \overline D\to\c\setminus\{0\}$ is a nowhere vanishing function in $\Ascr(D)$. Let $\phi_t\colon \overline \d\to \Agot\setminus\{0\}$ be the unique map of class $\Ascr(\d)$ determined by the condition 
\begin{equation}
\label{eq:phit}
	(\phi_t \circ h)\,\alpha = f_t \quad\text{on}\ \overline D,\quad \forall t\in B. 
\end{equation}
The chain rule then gives $h^*\left( \phi_t(\zeta)\, d\zeta\right) =f_t\,\theta= dF_t$ on $\overline D$ for all $t\in B$. 

We now follow the construction in the proof of Lemmas \ref{lem:RH} and \ref{lem:RHfamily}. Since $B\times \overline \d$ is simply connected, we can find a $\pi$-lifting of $\{\phi_t\}_{t\in B}$ to $\c^2\setminus \{(0,0)\}$, i.e., a holomorphic family of maps $(u_t,v_t) \colon \overline\d \to  \c^2\setminus \{(0,0)\}$ such that $\phi_t=\pi(u_t,v_t)$ for every $t\in B$; cf.\ \eqref{eq:F'}. Likewise we choose a vector $(p,q)\in \c^2\setminus\{(0,0)\}$ such that $\pi(p,q)=\vartheta$ is the given null direction vector for the arc $I$. We may consider the size function $\mu \colon C\to \r_+$ as a function on $b\d$ which vanishes outside the arc $I$. Set $\eta=\sqrt{\mu}\colon b{\d}\to \r_+$ and let $\tilde\eta$ denote its rational approximation of the form \eqref{eq:eta} with a pole of order $m-1$ at the origin. For every $n \ge m$ we define functions $u_{t,n}, v_{t,n} \colon \overline\d\to\c$ of class $\Ascr(\d)$ by \eqref{eq:tildef}:
\begin{equation}
	u_{t,n}(\xi) =  u_t(\xi)+\sqrt{2n+1} \,\tilde\eta(\xi) \, \xi^n p, \quad 
	v_{t,n}(\xi) =  v_t(\xi)+\sqrt{2n+1} \,\tilde\eta(\xi) \, \xi^n q.
\end{equation}
By the same argument as in the proof of Lemma \ref{lem:RH} we can change the vector $(p,q)\in\c^2$ slightly to ensure that for all large enough $n$ and for all $t$ sufficiently close to $0\in\c^N$ the image of the map $(u_{t,n},v_{t,n})\colon\overline\d\to \c^2$ avoids the origin $0\in\c^2$. By shrinking the ball $B$ around $0\in\c^N$ we may assume that this holds for all $t\in B$. For such choices of $n$ and $B$ the map $g_{t,n}:=\pi(u_{t,n},v_{t,n})\colon \overline\d \to\Agot$ has range in $\Agot\setminus\{0\}$, and hence it furnishes a null holomorphic immersion $G_{t,n} \in\IA(\d)$ by the expression
\begin{equation}
\label{eq:Gtn}
	G_{t,n}(\zeta)= F(x_0) + \int_0^\zeta g_{t,n}(\xi)\, d\xi,  \quad \zeta\in \overline\d,\ t\in B.
\end{equation}
(Compare with \eqref{eq:Gt}.) 
The dependence is holomorphic in $t\in B$. The estimates in the proof of Lemma \ref{lem:RH} show that, after a slight shrinking of the ball $B$, the sequence of null immersions 
\[
	\wt G_{t,n}:=G_{t,n}\circ h \colon \overline D\to \c^3,\quad  t\in B,\ n\in \n,
\]
satisfies Theorem \ref{th:RH} with respect to the null curve $F_t\colon \overline D\to\c^3$ and the size function $\mu$. In particular, as $n\to +\infty$, we have $\wt G_{t,n} \to F_t$ in the $\Cscr^1$ topology on $\overline D\setminus I$, uniformly on compacts and uniformly with respect to $t\in B$. 

Note that $d \wt G_{t,n} = \tilde g_{t,n} \,\theta$ where $\tilde g_{t,n}=(g_{t,n}\circ h)\alpha \colon \overline D\to \Agot\setminus \{0\}$.  We also have $f_t=(\phi_t \circ h)\,\alpha $ on $\overline D$ by \eqref{eq:phit}. Since $g_{t,n}\to \phi_t$ on $\overline D\setminus I$ as $n\to +\infty$, it follows that for any neighborhood $W\subset \overline D$ of the arc $I$ and for sufficiently big $n$ the map $\tilde g_{t,n}$ is arbitrarily $\Cscr^0$-close to $f_t$ on $\overline D\setminus W$, uniformly with respect to the parameter $t\in B$. 

To complete the proof, we will show that for all sufficiently big $n$ the two sprays of maps $\{f_t\}_{t\in B}$ and $\{\tilde g_{t,n}\}_{t\in B}$, with values in $\Agot\setminus \{0\}$, can be glued into a single spray $\tilde f_{t,n} \colon \overline M\to \Agot\setminus\{0\}$ of maps of class $\Ascr(M)$ such that $\tilde f_{t,n}$ is uniformly close to $f_t$ on $\overline {M\setminus D}$, and is uniformly close to $\wt g_{t,n}$ on $\overline D$. Assuming for a moment that this gluing can be done, let us indicate how it will be used to complete the proof of the theorem. Since the period map $B\ni t\mapsto \Pcal(f_t) \in\c^{3l}$ \eqref{eq:P} has maximal rank $3l$ at $t=0$ and all curves $\Gamma_j$ are contained in $\overline {M\setminus D}$ where $\tilde f_{t,n}$ is close to $f_t$ for big $n$, there exists a $t_0=t_0(n)\in B$ very close to $0$ (depending on how big $n$ we choose) such that the map $\tilde f_{t_0,n}\colon \overline M \to \Agot\setminus \{0\}$ has vanishing periods, so it integrates to a null curve $G:\overline M\to \c^3$. We shall verify that $G$ satisfies the special case of Theorem  \ref{th:RH} provided that all approximations were close enough and $n$ is chosen sufficiently big. 

We glue the sprays by the method in \cite[\S 4]{DF2007} (see also \cite[\S 5.9]{F2011}). From now on the parameter ball $B\subset \c^N$ and the size function $\tilde \eta$ will be kept fixed; the necessary approximations to enable gluing and to control the subsequent estimates will be achieved by choosing the integer $n$ big enough. We begin by explaining the underlying geometry. 

Recall that the arc $I \subset C$ is contained in the relative interior of $I'=bD\cap bM=bM\cap C$ and the function $\mu\colon bD\to \r_+$ vanishes on $bD\setminus I$. Also, $A\subset \overline M$ is an annular neighborhood of the boundary curve $C\subset bM$. Choose a smoothly bounded, simply connected compact neighborhood $V\subset \overline D \cap A$ of $I$ in $\overline D$ such that $bV\cap bM$ is contained in the relative interior of $I'$. Note that the arc $I$ is contained in the relative interior of $bV\cap bM$. 

By denting the boundary curve $C\subset bM$ slighty inwards along a relative neighborhood of the arc $I$ we obtain a compact, connected, smoothly bounded domain $U\subset \overline M$ such that $U\cup V=\overline M$, the intersection $U\cap V$ has smooth boundary, and $\overline{U\setminus  V}\cap \overline{V\setminus  U} =\emptyset$. A pair $(U,V)$ with these properties  is called a {\em Cartan decomposition} of the bordered Riemann surface $\overline M$ (cf.\ \cite[\S 5.7]{F2011}). We can also ensure that the point $x_0$ lies in $U$. See Fig.\ \ref{fig:UV}.
\begin{figure}[ht]
    \begin{center}
    \resizebox{0.8\textwidth}{!}{\includegraphics{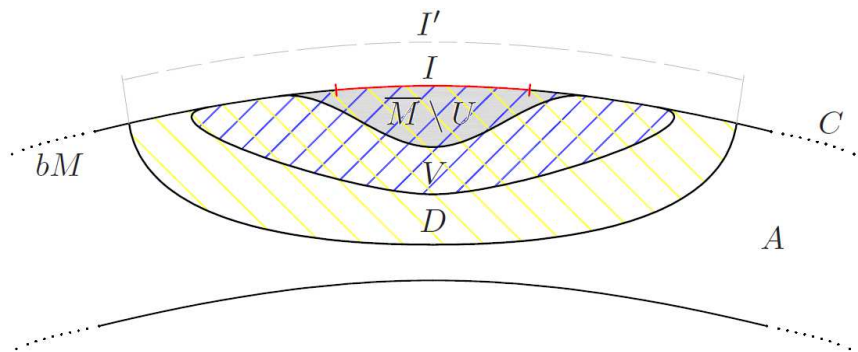}}
    \end{center}
        \vspace{-0.25cm}
\caption{The sets $U$ and $V$.}\label{fig:UV}
\end{figure} 

We now consider the following two sprays of maps with values in $\Agot\setminus\{0\}$. The first spray over $U$ is $\{f_t|_U \colon t\in B\}$; this spray is of class $\Ascr(U)$ and is dominating by the construction. The second spray over $V$ is $\{\tilde g_{t,n} \colon t\in B\}$; it is of class $\Ascr(V)$. Recall that the second spray converges to the first one uniformly on $B\times  (U\cap V)$ as $n\to +\infty$. Choose a slightly smaller ball $0\in B'\Subset B$ and a number $\delta>0$. By \cite[Proposition 4.3]{DF2007} (or \cite[Proposition 5.9.2, p.\ 216]{F2011}) there exists for every sufficiently big $n\in \n$ a pair of holomorphic maps $\alpha_n\colon B'\times U\to B$, $\beta_n \colon B'\times V\to B$, of the form
\[
		\alpha_n(t,x)=t+a_n(t,x), \quad \beta_n(t,x)=t+b_n(t,x), 
\]
satisfying the following sup-norm estimates
\begin{equation}
\label{eq:est-ab}
		\|a_n\|_{0,B'\times U} <\delta, \quad \|b_n\|_{0,B'\times V} <\delta,
\end{equation}
and the compatibility condition  
\begin{equation}
\label{eq:compat}
		f_{\alpha_n(t,x)}(x) = \wt g_{\beta_n(t,x),n}(x),  \quad x\in U\cap V,\ t\in B'.
\end{equation}
Hence the two sides above amalgamate into a spray of maps $\tilde f_{t,n} \colon U\cup V=\overline M\to \Agot\setminus\{0\}$ of class $\Ascr(M,\c^3)$ depending holomorphically on the parameter $t\in B'$. 

By the argument made above, we can find for every sufficiently big $n\in\n$ a point $t_0=t_0(n) \in B'$ arbitrarily near $0$ such that the map $\tilde f_{t_0,n}$ has vanishing periods. Let
\[
		F_n(x)= F(x_0) + \int_{x_0}^x \tilde f_{t_0,n}\,\theta, \qquad x \in \overline M.
\]
Then the map $G:=F_n$ satisfies the conclusion of Theorem \ref{th:RH} if $\delta>0$ is small enough and $n\in \n$ is big enough. Indeed, on the set $U$, $F_n$ is as close as desired to $F$ in the $\Cscr^1(U)$ topology for big $n$ as follows from \eqref{eq:est-ab} (applied with a sufficiently small $\delta>0$) and \eqref{eq:compat}. Over $V$ we can pass to the coordinate $\zeta=h(x)\in\overline\d$ $(x\in V)$, and the desired estimates follow immediately from Lemma \ref{lem:RHfamily} (the case of variable parameter).
\end{proof}


\section{On the conformal Calabi-Yau problem}
\label{sec:th}

In this section we prove Theorem \ref{th:intro} stated in the introduction; this will also prove Corollary \ref{co:r3}. Theorem \ref{th:intro} is a particular case of the following more precise result.

\begin{theorem}\label{th:main}
Let $M$ be a bordered Riemann surface. Every null holomorphic immersion $F\in\IA(M)$ can be approximated, uniformly on compacta in $M$, by complete proper 
null holomorphic embeddings into any open ball in $\c^3$ containing the image $F(\overline M)$.
\end{theorem}

As usual in such constructions, Theorem \ref{th:main} follows from a recursive application of an approximation result, Lemma \ref{lem:main} below. 

We denote by $\b(s)=\{u\in\c^3\colon |u|<s\}$ the open ball of radius $s>0$ centered at the origin in $\c^3$, and by $\overline{\b}(s)=\{u\in\c^3\colon |u|\leq s\}$ the closed ball with same center and radius. Recall  from \S \ref{sec:prelim} that $\Bscr( M)$ is the set of all bordered Runge domains $\Rcal\Subset M$ homeotopic to $M$, and $\IA( M)$ is the set of all $\Cscr^1$ null immersions $\overline  M\to \c^3$ that are holomorphic on $M$.

%
%
%
%
\begin{lemma}\label{lem:main}
Let $ M$ be a bordered Riemann surface, let $\Rcal\in\Bscr( M)$, let $z_0$ be a point in $\Rcal$, let $F\in \IA( M)$, and let $\epsilon$, $\rho$, and $s>\epsilon$ be positive numbers. Assume that
\begin{enumerate}[\rm (i)]
\item $F(\overline{ M} \setminus  \Rcal)\subset \b(s) \setminus  \overline{\b}(s-\epsilon)$, and
\item $\dist_{(\overline{M},F)}(z_0,b M)>\rho$.
\end{enumerate}
Then, for any $\hat{\epsilon}>0$ and $\delta>0$, there exists $\wh F \in \IA( M)$ 
enjoying the following properties: 
\begin{enumerate}[\rm (L1)]
\item $\|\wh F -F\|_{1,\overline{\Rcal}}<\hat{\epsilon}$.
\item $\wh F (bM)\subset \b(\hat{s}) \setminus  \overline{\b}(\hat{s}-\hat{\epsilon})$, where $\hat{s} :=\sqrt{s^2+\delta^2}$.
\item $\wh F (\overline{ M}\setminus  \Rcal)\cap\overline{\b}(s-\epsilon)=\emptyset$.
\item $\dist_{(\overline{M},\wh F )}(z_0,b M)>\hat{\rho} :=\rho+\delta$.
\end{enumerate}
\end{lemma}

Lemma \ref{lem:main} is the analogue of \cite[Lemma 5]{AF1} in the context of null curves in $\c^3$. The main improvement with respect to previous related constructions of null curves in $\c^3$ and minimal surfaces in $\r^3$ is that we do not modify the complex structure on the curve; compare for instance with \cite[Lemma 5]{AFM}, \cite[Lemma 3]{AL2}, and \cite[Lemma 4.1]{AL-Israel}, among others. 


\begin{proof}[Proof of Theorem \ref{th:main} assuming Lemma \ref{lem:main}]
The proof is an adaptation of the one of \cite[Theorem 2]{AF1}; we indicate the main points and refer to the cited source for further details.

Let $ M$ be a bordered Riemann surface, let $F\in\IA( M)$, and let $B\subset\c^3$ be an open ball containing $F(\overline{ M})$.  Fix a compact set $K\subset  M$ and a positive number $\eta>0$. By a translation and a homothety we assume that $B=\b$ is the unit ball and that $F(b M)$ does not contain the origin of $\c^3$. Moreover, by general position we may assume that $F\colon\overline M\to\c^3$ is an embedding; see \cite[Theorem 2.4]{AF2}.

Pick numbers $0< \varepsilon < r<1$ such that $F(b M) \subset \b(r)\setminus \overline{\b}(r-\varepsilon)$. Choose a bordered domain $\Rcal\in\Bscr( M)$ such that  
$K\subset \Rcal$  and $F(\overline{ M} \setminus  \Rcal)\subset \b(r)\setminus \overline{\b}(r-\varepsilon)$.
Set $c :=\sqrt{6(1-r^2)}/\pi>0$. Fix a point $\zeta_0\in \Rcal$, and define sequences 
$\rho_n,r_n>0$ $(n\in \n\cup\{0\})$ recursively as follows:  
\[
   \rho_0=\dist_{( M,F)}(\zeta_0,b \Rcal),\enskip
   \rho_n=\rho_{n-1}+\frac{c}{n},\quad r_0=r, \enskip r_n=\sqrt{r_{n-1}^2+\frac{c^2}{n^2}}.
\]
Notice that
\begin{equation}\label{eq:diverges-converges}
\lim_{n\to+\infty}\rho_n=+\infty \quad\text{and}\quad \lim_{n\to+\infty}r_n=1.
\end{equation}

Set $(\Rcal_0,F_0,\varepsilon_0):=(\Rcal,F,\varepsilon)$. We shall inductively construct sequences of bordered domains $\Rcal_n\in\Bscr( M)$, null holomorphic embeddings $F_n\in\IA( M)$, and constants $\varepsilon_n>0$  satisfying the following conditions for every $n\in\n$:
\begin{enumerate}[\rm (a$_n$)]
\item $\Rcal_{n-1}\Subset \Rcal_n$.
\item $\displaystyle \varepsilon_n <\min\Big\{ \frac{\varepsilon_{n-1}}2 \,,\, \frac1{2^n}\min\big\{ \min_{\overline{\Rcal}_j} |dF_j|\colon j=1,\ldots,n-1 \big\} \Big\}$ and every holomorphic map $G\colon \overline{ M}\to\c^3$ satisfying $\|G-F_{n-1}\|_{1,\overline{\Rcal}_{n-1}}<2\varepsilon_n$ is an embedding on $\overline{\Rcal}_{n-1}$.
\item $\|F_n-F_{n-1}\|_{1,\overline{\Rcal}_{n-1}}<\varepsilon_n$.
\item $F_n(\overline{ M}\setminus \Rcal_n)\subset \b(r_n)\setminus \overline{\b}(r_n-\varepsilon_n)$.
\item $F_n(\overline{ M}\setminus \Rcal_{n-1})\cap \overline{\b}(r_{n-1}-\varepsilon_{n-1})=\emptyset$.
\item $\dist_{( M,F_n)}(\zeta_0,b \Rcal_n)>\rho_n$.
\end{enumerate}
We will also ensure that 
\begin{equation}\label{eq:cupN}
	 M=\bigcup_{n\in\n}\Rcal_n.
\end{equation}

For the basis of the induction, take a number $0<\varepsilon_1<\eta/2$ small enough so that {\rm (b$_1$)} holds; such $\varepsilon_1$ exists since $F_0\in\IA( M)$ is an embedding. Applying Lemma \ref{lem:main} to the data
\[
(M\,,\, \Rcal \,,\, z_0 \,,\, F \,,\, \epsilon \,,\, \rho \,,\, s \,,\, \hat{\epsilon} \,,\, \delta)= (M\,,\,\Rcal_0 \,,\, \zeta_0 \,,\, F_0 \,,\, \varepsilon_0 \,,\, \rho_0 \,,\,r_0 \,,\, \varepsilon_1 \,,\,c),
\]
one gets a null holomorphic immersion $F_1\in\IA( M)$, which can be assumed to be an embedding by general position argument \cite{AF2}. It is straightforward to check that {\rm (a$_1$)}--{\rm (f$_1$)} are satisfied up to fixing a sufficiently large bordered domain $\Rcal_1\in\Bscr( M)$.

For the inductive step, assume that for some $n>1$ we already have $(\Rcal_{j},h_{j},\varepsilon_{j})$ enjoying properties {\rm (a$_j$)}--{\rm (f$_j$)} for all $j\in\{1,\ldots,n-1\}$. Since $F_{n-1}\colon \overline{ M}\to\c^3$ is an embedding, there exists a number $\varepsilon_n>0$ satisfying {\rm (b$_n$)}.
Applying Lemma \ref{lem:main} to the data
\[
(M\,,\ \Rcal \,,\, z_0 \,,\, F \,,\, \epsilon \,,\, \rho \,,\, s \,,\, \hat{\epsilon} \,,\, \delta)= (M\,,\, \Rcal_{n-1} \,,\, \zeta_0 \,,\, F_{n-1} \,,\, \varepsilon_{n-1} \,,\, \rho_{n-1} \,,\,r_{n-1} \,,\, \varepsilon_n \,,\,\frac{c}{n}),
\]
one obtains a null holomorphic immersion $F_n\in\IA( M)$ which can be assumed an embedding by general position \cite{AF2} and, together with $\varepsilon_n$ and a sufficiently large domain $\Rcal_n \in \Bscr( M)$, meets all requirements {\rm (a$_n$)}--{\rm (f$_n$)}. 

Finally, \eqref{eq:cupN} holds if the bordered domain $\Rcal_n$ is chosen large enough in each step of the inductive process. This concludes the construction of the sequence $\{(\Rcal_n,F_n,\varepsilon_n)\}_{n\in\n}$.

Since $F_n\in\IA( M)$ for all $n\in\n$, properties {\rm (c$_n$)}, {\rm (b$_n$)}, and (\ref{eq:cupN}) ensure that the sequence $\{F_n\colon \overline { M}\to\c^3\}_{n\in\n}$ converges uniformly on compacta in $ M$ to a holomorphic map $\wh F \colon  M\to\c^3$ whose derivative, $\wh F'$, with respect to any local holomorphic coordinate on $M$ assumes values in the null quadric $\Agot$. 

Let us show that $\wh F$ satisfies the conclusion of Theorem \ref{th:main}. Indeed, from {\rm (b$_n$)} and {\rm (c$_n$)} we have that
\begin{equation}\label{eq:hatFclose}
\|\wh F-F_j\|_{1,\overline{\Rcal}_j}<2\varepsilon_{j+1}<\varepsilon_j<\eta/2^j\quad \forall j\in\n\cup\{0\};
\end{equation}
recall that $\varepsilon_1<\eta/2$. This and {\rm (c$_n$)} imply that $\wh F\colon M\to\c^3$ is a null holomorphic embedding which is $\eta$-close to $F=F_0$ in the $\Cscr^1$ topology on $\Rcal_0\supset K$. To finish, it suffices to check that $\wh F( M)\subset\b$ and $\wh F\colon M\to\b$ is complete and proper.

For the inclusion in $\b$, let $p\in M$. From {\rm (d$_n$)} and the Maximum Principle we have $|F_n(p)|<r_n$ for all $n\in\n$. By the second part of \eqref{eq:diverges-converges} and taking limits as $n\to+\infty$, one obtains  $|\wh F(p)|\leq 1$; hence $\wh F( M)\subset\b$ by the Maximum Principle.

Next we show that $\wh F\colon  M\to\b$ is a proper map. Pick any number $0<t<1$. Since $r_n\to 1$ and $\varepsilon_n\to 0$ as $n\to+\infty$, there exists $n_0\in\n$ large enough so that 
\begin{equation}\label{eq:r-t}
		t+\varepsilon_{n-1}+\eta/2^n<r_{n-1}\quad \forall n\geq n_0.
\end{equation}
On the other hand, {\rm (e$_n$)} and \eqref{eq:hatFclose} give that
\[
		\wh F(\overline{\Rcal}_n\setminus \Rcal_{n-1})\cap 			
	  \overline{\b}(r_{n-1}-\varepsilon_{n-1}-\eta/2^n)=\emptyset\quad \forall n\geq n_0.
\]
Together with \eqref{eq:r-t} this implies  
\[
		(\overline{\Rcal}_n\setminus \Rcal_{n-1})\cap \wh F^{-1}(\overline{\b}(t))
		=\emptyset\quad \forall n\geq n_0;
\]
hence $\wh F^{-1}(\overline{\b}(t))\subset \Rcal_{n_0}$ is compact in $ M$. 
This shows that $\wh F\colon  M\to\b$ is a proper map.

Finally, to verify completeness, let $p\in  M$ and $n_0\in\n$ be such that $p\in\Rcal_{n_0}$. From {\rm (c$_n$)} and {\rm (b$_n$)} it is not hard to infer that $2|d\wh F(p)|>|dF_{n_0}(p)|$; hence
\[
\dist_{( M,\wh F)}(\zeta_0,b\Rcal_n)>\frac12\,\dist_{( M,F_n)}(\zeta_0,b\Rcal_n) \stackrel{\text{\rm (f$_n$)}}{>} \frac{\rho_n}2\quad\forall n\in\n.
\]
Taking limits in the above inequality as $n\to+\infty$, the completeness of $\wh F$ follows from the former equation in \eqref{eq:diverges-converges}.

This concludes the proof of Theorem \ref{th:main} granted Lemma \ref{lem:main}.
\end{proof}


\begin{proof}[Proof of Lemma \ref{lem:main}]
The proof consists of adapting to the context of null holomorphic curves in $\c^3$ the technique that we developed in order to prove Lemma 5 in our previous paper \cite{AF1}. We will therefore emphasize the main differences with respect to the argument in \cite{AF1} and will omit some of the details which can be found in that paper. 

We begin with a couple of reductions. Replacing $\Rcal$ by a larger bordered domain in $\Bscr( M)$ if necessary, we assume without loss of generality that 
\begin{equation}\label{eq:lemma}
		\dist_{( M,F)}(z_0,b\Rcal)>\rho.
\end{equation}
We may assume that $M$ is a bordered domain in an open Riemann surface $\widehat{ M}$ such that $ M \in \Bscr(\widehat{M})$. By the Mergelyan theorem for null curves (see \cite[Lemma 1]{AL2} or \cite[Theorem 7.2]{AF2}) we may also assume that $F\in\IA( M)$ extends to a null holomorphic immersion $F\in\IA(\wh  M)$.

The proof consists of three different steps. The main novelties appear in the last one where the approximate solution to the Riemann-Hilbert problem for null curves (cf.\ Theorem \ref{th:RH} in \S\ref{sec:RH}) is invoked.


\subsubsection*{Step 1: Splitting $bM$}\label{sec:step1}
The first step in the proof consists of suitably splitting the boundary curves of $M$ into a finite collection of compact Jordan arcs. The splitting strongly depends on the placement in $\c^3$ of the image $F(bM)$ of the boundary of $M$.

Up to decreasing the number $\hat{\epsilon}>0$ if necessary, it follows from {\rm (i)}, the definition of $\hat s$ in {\rm (L2)}, the strict convexity and compactness of $\overline{\b}(s-\epsilon)$, and Pythagoras' theorem, that every point $u\in F(b M)$ admits an open neighborhood $\Uscr_u$ in $\c^3$ such that
\begin{equation}\label{eq:Pytha}
\dist\big( v\,,\, (\b(\hat{s}) \setminus  \overline{\b}(\hat{s}-\hat{\epsilon}))\cap (w+\langle z\rangle^\bot)\big)>\delta\quad \forall v,w,z\in \Uscr_u
\end{equation}
and
\begin{equation}\label{eq:far}
(v+\langle w\rangle^\bot)\cap\overline{\b}(s-\epsilon)=\emptyset\quad \forall v,w\in \Uscr_u.
\end{equation}
Furthermore, since $F(b M)\subset \b(s)\Subset\b (\hat s)$ and these balls are centered at the origin of $\c^3$, we may assume (after possibly shrinking the sets $\Uscr_u$) that for every point $u\in F(bM)$ there exists a real constant $\tau=\tau(u)>0$ such that
\begin{equation}\label{eq:sphere}
v+(\langle w\rangle^\bot \cap \{z\in\c^3\colon |z|=\tau\}) \subset \b(\hat s)\setminus\overline \b (\hat s-\hat\epsilon)\quad \forall v,w\in \Uscr_u.
\end{equation}

Set $\Uscr:=\{\Uscr_u\colon u\in F(b M)\}.$

Denote by $ C_1,\ldots, C_\igot$ the connected boundary curves of $\overline{ M}$; these are pairwise disjoint smooth closed Jordan curves in $\widehat{M}$ and $bM=\bigcup_{i=1}^\igot  C_i$. Since $\Uscr$ is an open covering of the compact set $F(b  M)\subset\c^2$, we may choose an integer $\jgot \geq 3$ and compact connected subarcs $\{ C_{i,j}\colon (i,j) \in \I:=\{1,\ldots,\igot\}\times \z_\jgot\}$ (here $\z_\jgot=\{0,\ldots,\jgot-1\}$ denotes the additive cyclic group of integers modulus $\jgot$), enjoying the following properties:
\begin{enumerate}[\rm ({a}1)]
\item $\bigcup_{j=1}^\jgot C_{i,j}= C_i$.
\item $ C_{i,j}$ and $ C_{i,j+1}$ have a common endpoint, namely $p_{i,j}$, and are otherwise disjoint.
\item There exist points $a_{i,j}\in \b(\hat{s})\setminus \overline{\b}(\hat{s}-\hat{\epsilon})$ such that 
\[
(a_{i,j}+\langle F(p_{i,k})\rangle^\bot)\cap \overline{\b}(\hat{s}-\hat{\epsilon})=\emptyset \quad \forall k\in\{j,j+1\}.
\] 
\item For every $(i,j) \in \I$ there exists a set $\Uscr_{i,j}\in\Uscr$ containing $F( C_{i,j})$. In particular $F(p_{i,j})\in \Uscr_{i,j}\cap\Uscr_{i,j+1}$.
\item For every point $u\in \span\{F(p_{i,j})\}\cap \overline\b(\hat s)$ there exists a constant $\tau>0$ such that
\[
u+(\langle F(p_{i,k})\rangle^\bot \cap \{z\in\c^3\colon |z|=\tau\}) \subset \b(\hat s)\setminus\overline \b (\hat s-\hat\epsilon)\quad \forall k\in\{j,j+1\}.
\]
\end{enumerate}

To find a splitting with these properties, simply choose the arcs $ C_{i,j}$ such that their images $F( C_{i,j}) \subset \c^2$ have small enough diameter.


\subsubsection*{Step 2: Stretching from the points $p_{i,j}$}\label{sec:step2}

We now describe how to attach to the immersed null curve $F(\overline M)\subset \c^3$ a family of compact arcs at the points $\{F(p_{i,j})\colon (i,j)\in \I\}$ and then stretch the image along these arcs, modifying the map only slightly away from the points $p_{i,j}$. The resulting null curve is still normalized by $\overline M$, and the boundary distance to the points $p_{i,j}$ increases by at least $\delta$ (see (L4)).

For every $(i,j)\in\I$ we choose an embedded real analytic arc $\gamma_{i,j}\subset\widehat M$ that is attached to $\overline M$ at the endpoint $p_{i,j}$ and intersects $b M$ transversely there, and is otherwise disjoint from $\overline M$. We take the arcs $\gamma_{i,j}$ to be pairwise disjoint. Let $q_{i,j}$ denote the other endpoint of $\gamma_{i,j}$. We split $\gamma_{i,j}$ into compact subarcs $\gamma_{i,j}^1$ and $\gamma_{i,j}^2$, with a common endpoint, so that $p_{i,j}\in\gamma_{i,j}^1$ and $q_{i,j}\in\gamma_{i,j}^2$.

Choose compact smooth embedded arcs $\lambda_{i,j} \subset \c^3$ satisfying the following properties:
\begin{enumerate}[\rm ({b}1)]
\item $\lambda_{i,j}\subset \b(\hat{s})\setminus \overline{\b}(s-\epsilon)$.
\item $\lambda_{i,j}$ is split into compact subarcs $\lambda_{i,j}^1$ and $\lambda_{i,j}^2$ with a common endpoint.
\item $\lambda_{i,j}^1$ agrees with the arc $F(\gamma_{i,j})$ near the endpoint $F(p_{i,j})$; recall that $F\in\IA(\wh M)$.
\item $\lambda_{i,j}^1\subset \Uscr_{i,j}\cap\Uscr_{i,j+1}$; see {\rm (a4)} and {\rm (b3)}.
\item If $J\subset\lambda_{i,j}^1$ is a Borel measurable subset, then
\begin{eqnarray*}
\min\{\length(\pi_{i,j}(J)),\length(\pi_{i,j+1}(J))\}&+&\\
\min\{\length(\pi_{i,j}(\lambda_{i,j}^1\setminus J)),\length(\pi_{i,j+1}(\lambda_{i,j}^1\setminus J))\}&>&\delta,
\end{eqnarray*}
where $\pi_{i,k}\colon \c^3\to\span\{F(p_{i,k})\}\subset\c^3$ denotes the orthogonal projection; observe that $F(p_{i,k})\neq 0$ by Lemma \ref{lem:main}-{\rm (i)}.
\item $(\lambda_{i,j}^2+\langle F(p_{i,k})\rangle^\bot)\cap \overline{\b}(s-\epsilon)=\emptyset$ for $k\in \{j,j+1\}$; see {\rm (a4)}, {\rm (b4)}, and \eqref{eq:far}.
\item The endpoint $v_{i,j}$ of $\lambda_{i,j}$ contained in the subarc $\lambda_{i,j}^2$ lies in $\span\{F(p_{i,j})\}\cap\big(\b(\hat{s})\setminus \overline{\b}(\hat{s}-\hat{\epsilon})\big)$ and satisfies $(v_{i,j}+\langle F(p_{i,k})\rangle^\bot)\cap \overline{\b}(\hat{s}-\hat{\epsilon})=\emptyset$ for $k\in\{j,j+1\}$; see {\rm (a3)}.
\item For every point $u\in \lambda_{i,j}^2$ there exists a constant $\tau>0$ such that
\[
u+(\langle F(p_{i,k})\rangle^\bot \cap \{z\in\c^3\colon |z|=\tau\}) \subset \b(\hat s)\setminus\overline \b (\hat s-\hat\epsilon)\quad \forall k\in\{j,j+1\}.
\]
\item The tangent vector to $\lambda_{i,j}$ at any point lies in $\Agot\setminus\{0\}$; see \eqref{eq:Agot} and take into account property {\rm (b3)}.
\end{enumerate}

To construct such arcs $\lambda_{i,j}$, we first chose arcs $\wh\lambda_{i,j}$ meeting all the above requirements except for {\rm (b8)} and {\rm (b9)}; such were provided in \cite[Sec.\ 3.2]{AF1}. In order to ensure also {\rm (b8)}, we simply take $\lambda_{i,j}^2$ to be close to a segment in $\span\{F(p_{i,j})\}$; take into account \eqref{eq:sphere} and {\rm (a5)}. Then we use Lemma 7.3 in \cite{AF2}, which is analogous to Gromov's convex integration lemma, in order to approximate the arcs $\wh\lambda_{i,j}$ by another ones $\lambda_{i,j}$ meeting also {\rm (b9)}.

Notice that the compact set 
\[
	K:=\overline{ M}\cup \bigcup_{(i,j)\in \I}\gamma_{i,j} \subset \widehat M
\]
is {\em admissible} in the sense of \cite[Def.\ 2.2]{AL1} (see also \cite[Def. 2]{AL2} or \cite[Def.\ 7.1]{AF2}). Furthermore, property {\rm (b3)} enables one to find a smooth map $F_{\rm e}\colon \widehat  M \to\c^3$ which agrees with $F$ on an open neighborhood of $\overline  M$, and which maps the arcs $\gamma_{i,j}^1$ and $\gamma_{i,j}^2$ diffeomorphically onto the corresponding arcs $\lambda_{i,j}^1$ and $\lambda_{i,j}^2$ for all $(i,j)\in\I$. Taking into account {\rm (b9)}, the map $F_{\rm e}\colon K\to \c^3$ is a {\em generalized null curve} in the sense of \cite[Def.\ 5]{AL2}. Since the set $K$ is admissible, the Mergelyan theorem for null holomorphic immersions \cite{AF2,AL2} applies, allowing one to approximate $F_{\rm e}$, uniformly on an open neighborhood of $\overline{ M}$ and in the $\Cscr^1$ topology on each of the arcs $\gamma_{i,j}$, by a null holomorphic immersion $\wt F\colon 
 \widehat  M \to\c^3$. 

Let $V\subset \widehat  M$ be a small open neighborhood of $K$. For every $(i,j)\in\I$ we choose a pair of small neighborhoods $W'_{i,j} \Subset W_{i,j} \Subset \widehat  M\setminus \overline \Rcal$ of the point $p_{i,j}$ and a neighborhood $V_{i,j} \Subset \widehat  M \setminus \overline \Rcal$ of the arc $\gamma_{i,j}$. (Here $\Rcal\in\Bscr(M)$ is the domain in the statement of the lemma.) In this setting, \cite[Theorem 2.3]{FW0} (see also \cite[Theorem 8.8.1]{F2011}) furnishes a smooth diffeomorphism $\phi\colon \overline { M}\to\phi (\overline{ M})\subset V$ satisfying the following properties (see Fig.\ \ref{fig:Phi}):
\begin{itemize}
\item $\phi \colon  M \to  \phi( M)$ is a biholomorphism.
\item $\phi$ is as close as desired to the identity in the $\Cscr^1$ topology on $\overline  M \setminus \bigcup_{(i,j)\in\I} W'_{i,j}$. 
\item $\phi(p_{i,j}) = q_{i,j}$ and $\phi(\overline  M \cap W'_{i,j}) \subset W_{i,j} \cup V_{i,j}$ for all $(i,j)\in\I$. 
\end{itemize}
Furthermore, up to slightly deforming each of the arcs $\gamma_{i,j}$, keeping its endpoints fixed, we can replace $\lambda_{i,j}$ by the arc $\wt F(\gamma_{i,j})$ preseving the properties {\rm (b1)}--{\rm (b9)} above; recall that $\wt F$ is null. Therefore, we can also assume  that
\begin{itemize}
\item $\gamma_{i,j}\setminus\{q_{i,j}\}\subset\phi( M \setminus \overline{\Rcal})$ for all $(i,j)\in\I$.
\end{itemize} 

Set 
\begin{equation} \label{eq:sigma}
\sigma_{i,j}^k=\phi^{-1}(\gamma_{i,j}^k)\; \text{for $k\in\{1,2\}$}, \qquad \sigma_{i,j}=\phi^{-1}(\gamma_{i,j})=\sigma_{i,j}^1\cup \sigma_{i,j}^2.
\end{equation}
(See Fig.\ \ref{fig:Phi}.)
\begin{figure}[ht]
    \begin{center}
    \resizebox{0.95\textwidth}{!}{\includegraphics{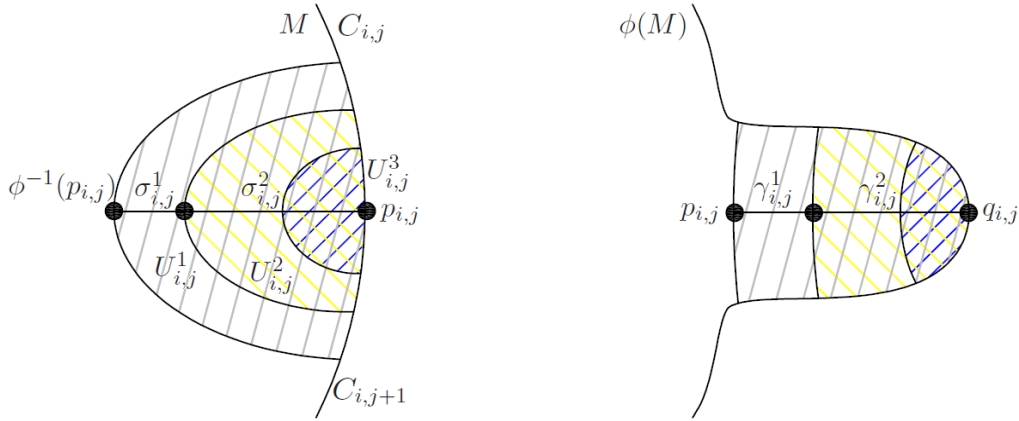}}
        \end{center}
        \vspace{-0.25cm}
\caption{The biholomorphism $\phi$.}\label{fig:Phi}
\end{figure} 
If one chooses the diffeomorphism $\phi$ close enough to the identity in the $\Cscr^1$ topology on $\overline  M \setminus \bigcup_{(i,j)\in\I} W'_{i,j}$, then the map
\begin{equation}
\label{eq:f0}
		F_0:=\wh F\circ\phi\colon \overline{ M}\to\c^3
\end{equation}
lies in $\IA( M)$ and enjoys the following properties:
\begin{enumerate}[\rm ({c}1)]
\item $F_0(p_{i,j})\in \b(\hat{s})\setminus \overline{\b}(\hat{s}-\hat{\epsilon})$ and $(F_0(p_{i,j})+\langle F(p_{i,k})\rangle^\bot)\cap \overline{\b}(\hat{s}-\hat{\epsilon})=\emptyset$ for $k\in\{j,j+1\}$ and $(i,j)\in\I$; see {\rm (b7)} and recall that $F_0(p_{i,j})=\wt{F}(q_{i,j})\approx F_{\rm e}(q_{i,j})=v_{i,j}$.
\item $\|F_0-F\|_{1,\overline{\Rcal}}<\hat{\epsilon}/2$.
\item $F_0(\overline{ M}\setminus \Rcal)\subset \b(\hat{s})\setminus \overline{\b}(s-\epsilon)$; see properties Lemma \ref{lem:main}-{\rm (i)} and {\rm (b1)}.
\item $\dist_{( M,F_0)}(z_0,b\Rcal)>\rho$; see \eqref{eq:lemma}.
\item For every point $p\in  C_{i,j}$, $(i,j)\in\I$, there exists a constant $\tau>0$, which can be taken to depend continuously on $p$, such that
\[
F_0(p)+(\langle F(p_{i,j})\rangle^\bot \cap \{z\in\c^3\colon |z|=\tau\}) \subset \b(\hat s)\setminus\overline \b (\hat s-\hat\epsilon);
\]
see \eqref{eq:sphere}, {\rm (b4)}, {\rm (b8)}, and \eqref{eq:sigma}.
\item $(F_0( C_{i,j})+\langle F(p_{i,j})\rangle^\bot)\cap \overline\b(s-\epsilon)=\emptyset$; see \eqref{eq:far}, {\rm (a4)}, {\rm (b4)}, and {\rm (b6)}.
\end{enumerate}

If the above approximations are close enough and the neighborhood $V$ of $K$ (containing the image $\phi(\overline  M)$) is small enough, then properties {\rm (b1)}--{\rm (b9)} and {\rm (c1)}--{\rm (c6)} allow one to find simply connected neigh\-bor\-hoods $U_{i,j}^3\Subset U_{i,j}^2\Subset U_{i,j}^1$ of the point $p_{i,j}$ in $\overline M\setminus \overline{\Rcal}$, $(i,j)\in\I$, meeting the following requirements  (see Figs.\ \ref{fig:Phi} and \ref{fig:R}):
\begin{enumerate}[\rm ({d}1)]
\item $\overline{U}_{i,j}^1\cap \overline{U}_{i,k}^1=\emptyset$ if $k\neq j$.
\item $\overline{U}_{i,j}^1\cap  C_{i,k}=\emptyset$ if $k\notin\{j,j+1\}$.
\item $\overline{U}_{i,j}^1\cap  C_{i,k}$ is a connected compact Jordan arc for  $k\in\{j,j+1\}$.
\item $\beta_{i,j}^k :=\overline{ C_{i,j}\setminus (U_{i,j}^{k+1}\cup U_{i,j-1}^{k+1})}$ are connected compact Jordan arcs for $k\in\{1,2\}$, and $F_0(\beta_{i,j}^1)\subset\Uscr_{i,j}$; see {\rm (a4)}.

\item $\sigma_{i,j}^1\subset \overline{U_{i,j}^1\setminus U_{i,j}^2}$, $\sigma_{i,j}^2\subset\overline{U}_{i,j}^2$ (see (\ref{eq:sigma})), and $\phi^{-1}(p_{i,j})\in (b {U}_{i,j}^1)\cap M$.

\begin{figure}[ht]
    \begin{center}
    \resizebox{0.9 \textwidth}{!}{\includegraphics{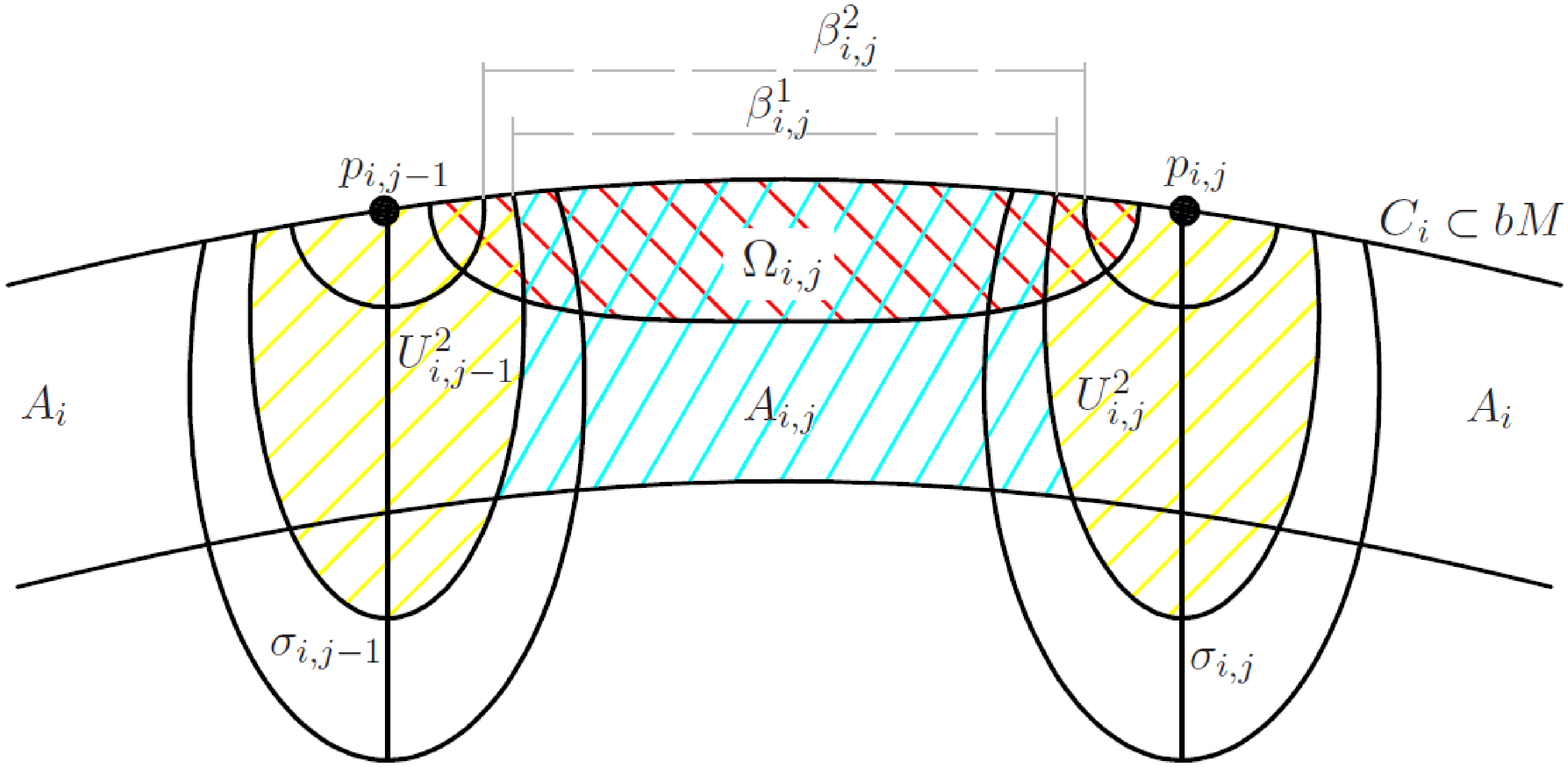}}
        \end{center}
        \vspace{-0.25cm}
\caption{The sets in $\overline M\setminus \overline{\Rcal}$.}\label{fig:R}
\end{figure}

\item $F_0(\overline{U}_{i,j}^3)\subset\b(\hat{s})\setminus \overline{\b}(\hat{s}-\hat{\epsilon})$ and $(F_0(\overline{U}_{i,j}^3)+\langle F(p_{i,k})\rangle^\bot)\cap \overline{\b}(\hat{s}-\hat{\epsilon})=\emptyset$ for  $k\in\{j,j+1\}$; see {\rm (c1)}.
\item $\left( F_0(\overline{U}_{i,j}^2) + \langle F(p_{i,k})\rangle^\bot \right) \cap \overline{\b}(s-\epsilon)=\emptyset$ for all $k\in \{j,j+1\}$; see {\rm (b6)}.
\item $F_0(\overline{U_{i,j}^1\setminus U_{i,j}^2})\subset \Uscr_{i,j}\cap\Uscr_{i,j+1}$; see {\rm (b4)}.
\item If $\gamma\subset \overline{U}_{i,j}^1$ is an arc connecting $ M\setminus U_{i,j}^1$ and $\overline{U}_{i,j}^2$, and $J\subset \gamma$ is a Borel measurable subset, then, see {\rm (b5)},
\begin{eqnarray*}
\min\{\length(\pi_{i,j}(F_0(J))),\length(\pi_{i,j+1}(F_0(J)))\}&+&\\
\min\{\length(\pi_{i,j}(F_0(\gamma \setminus J))),\length(\pi_{i,j+1}(F_0(\gamma\setminus J)))\}&>&\delta.
\end{eqnarray*}
\end{enumerate}


\subsubsection*{Step 3: Stretching from the arcs $ C_{i,j}$}\label{sec:step3}

In the final step of the proof we use approximate solutions of certain Riemann-Hilbert problems for null curves in $\c^3$, provided by Theorem \ref{th:RH}, in order to stretch the images of the central parts of the arcs $ C_{i,j}$ very close to the sphere of radius $\hat s$ in $\c^3$.

For every $i=1,\ldots,\igot$, we choose a small compact annular neighborhood $A_i\subset\overline M\setminus\overline\Rcal$ of the boundary curve $ C_i$ and a smooth retraction 
$\rho_i\colon A_i\to C_i$, such that $A_i\cap \overline U_{i,j}^2$ is a compact Jordan arc and $A_i\setminus \bigcup_{j=1}^\jgot U_{i,j}^2$ consists of $\jgot$ pairwise disjoint closed discs $A_{i,j}$, $j\in\z_\jgot$. Assume that $A_{i,j}\cap C_{i,k}\neq \emptyset$ if and only if $j=k$. (See Fig.\ \ref{fig:R}.) From {\rm (d4)} we may choose $A_i$ so that $\rho_i(A_{i,j})=\beta_{i,j}^1$ and
\begin{equation}
\label{eq:Aij}
F_0(A_{i,j})\subset\Uscr_{i,j}\quad \forall (i,j)\in\I.
\end{equation}

For every $(i,j)\in\I$, take a unitary null vector 
\begin{equation}\label{eq:vartij}
\vartheta_{i,j}\in (\Agot\setminus\{0\})\cap \langle F(p_{i,j})\rangle^\bot;
\end{equation}
such exists since $\langle F(p_{i,j})\rangle^\bot$ is a complex vectorial hyperplane of $\c^3$; recall that $F(p_{i,j})\neq 0$ by Lemma \ref{lem:main}-{\rm (i)}. Combining {\rm (c5)}, {\rm (c6)}, {\rm (d6)}, and \eqref{eq:vartij}, we can easily find a continuous function $\mu\colon b M\to \r_+$ such that 
\begin{equation}\label{eq:muij}
\mu=0\quad \text{on $bM\setminus \bigcup_{(i,j)\in\I}\beta_{i,j}^2$},
\end{equation}
and the map
\begin{equation}\label{eq:chi}
	\varkappa\colon bM \times\overline{\d}\to\c^3,\quad 
	\varkappa(p,\xi)=\left\{\begin{array}{ll}
	F_0(p); & p \in bM\setminus \bigcup_{(i,j)\in\I} \beta_{i,j}^2 \\
	F_0(p)+\mu(p)\,\xi\,\vartheta_{i,j}; & p\in  \beta_{i,j}^2,\; (i,j)\in\I,
	\end{array}\right.
\end{equation}
satisfies the following conditions:
\begin{enumerate}[\rm ({e}1)]
\item $\varkappa(p,b \d)\subset \b(\hat s)\setminus \overline\b(\hat s-\hat\epsilon)$ for all $p\in bM$.
\item $\varkappa(p,\overline\d)\cap \overline\b(s-\epsilon)=\emptyset$ for all $p\in bM$.
\end{enumerate}
Since the compact arcs $\beta_{i,j}^2$, $(i,j)\in\I$, are pairwise disjoint, the continuity of $\mu$ and \eqref{eq:muij} ensure that $\varkappa$ is well defined and continuous.

In this setting, Theorem \ref{th:RH} gives small open neighborhoods $\Omega_{i,j}\subset A_i$ of the arcs $\beta_{i,j}^2$ and a holomorphic null immersion $\wh F\in\IA(M)$ satisfying the following conditions:
\begin{enumerate}[\rm ({f}1)]
\item $\Omega_{i,j}\cap (\sigma_{i,j-1}\cup\sigma_{i,j})=\emptyset$ and $b \Omega_{i,j}\cap bA_i$ is a compact arc contained in the relative interior of $ C_{i,j}$; see Fig.\ \ref{fig:R}.
\item $\wh F(p)\in \b(\hat{s})\setminus \overline{\b}(\hat{s}-\hat{\epsilon})$ for all $p\in bM$; see {\rm (e1)} and Theorem \ref{th:RH}-{\it i)}. In particular, $\wh F(\overline M)\subset \b(\hat s)$ by the Maximum Principle.
\item $\pi_{i,j}\circ \wh F$ is close to $\pi_{i,j}\circ F_0$ on $\overline \Omega_{i,j}$; see \eqref{eq:chi}, \eqref{eq:vartij}, and Theorem \ref{th:RH}-{\it ii)}.
\item $\wh F$ is close to $F_0$ in the $\Cscr^1$ topology on $M':=\overline M\setminus \bigcup_{(i,j)\in\I} \Omega_{i,j}$; see Theorem \ref{th:RH}-{\it ii)}, {\it iii)}.
\item $\wh F(\overline M\setminus \Rcal)\cap \overline{\b}(s-\epsilon)=\emptyset$; see {\rm (c3)}, {\rm (e2)}, and Theorem \ref{th:RH}-{\it ii)}, {\it iii)}.
\end{enumerate}

By the Mergelyan theorem for null curves  \cite{AL2,AF2} we may assume that $\wh F$ extends (with the same name) to a null holomorphic immersion of a neighborhood of $\overline  M$ in $\widehat  M$. 


Let us verify that $\wh F$ meets all the requirements of Lemma \ref{lem:main} provided that the approximations in {\rm (f3)} and {\rm (f4)} are close enough. 

Property {\rm (L1)} follows from {\rm (c2)} and {\rm (f4)}; recall that $\overline\Rcal\subset M'$. Properties {\rm (L2)} and {\rm (L3)} agree with {\rm (f2)} and {\rm (f5)}, respectively. 

In order to verify property {\rm (L4)} we have to work a little harder. In view of {\rm (f4)} and {\rm (c4)}, it suffices to check that $\dist_{(\overline M,\wh F)}(b\Rcal,b M)>\delta$. Let $\gamma$ be any arc in $\overline M\setminus \Rcal$ with one endpoint in $b\Rcal$ and the other one in $bM$. We wish to see that  $\length(\wh F(\gamma))>\delta$. 

Assume first  that $\gamma\cap \overline{U}_{i,j}^2\neq\emptyset$ for some $(i,j)\in\I$. In this case there exists a subarc $\hat \gamma \subset \overline{U}_{i,j}^1$ of $\gamma$ connecting $M\setminus U_{i,j}^1$ and $\overline{U}_{i,j}^2$; see Fig.\ \ref{fig:R}. Then
\begin{eqnarray*}
\length(\wh F(\hat{\gamma}))& \stackrel{\text{\rm (f4)}}\approx &  \length (\wh F(\hat{\gamma}\cap \overline{\Omega}_{i,j}))\ +\ \length (\wh F(\hat{\gamma}\cap \overline{\Omega}_{i,j+1}))
\\
& &  +\ \length (F_0(\hat{\gamma}\cap M'))\\
& \geq &  \length (\pi_{i,j}(\wh F(\hat{\gamma}\cap \overline{\Omega}_{i,j})))\ +\ \length (\pi_{i,j+1}(\wh F(\hat{\gamma}\cap \overline{\Omega}_{i,j+1})))
\\
& & +\ \length(F_0(\hat{\gamma}\cap M'))\\
& \stackrel{\text{\rm (f3)}}\approx & \length (\pi_{i,j}(F_0(\hat{\gamma}\cap \overline{\Omega}_{i,j})))\ +\ \length (\pi_{i,j+1}(F_0(\hat{\gamma}\cap \overline{\Omega}_{i,j+1}))) 
\\
& &  +\ \length (F_0(\hat{\gamma}\cap M')) \stackrel{\text{\rm (d9)}}> \delta;
\end{eqnarray*} 
hence $\delta< \length(\wh F(\hat{\gamma}))\leq \length(\wh F(\gamma))$ and we are done.

It remains to consider the case when $\gamma\cap \bigcup_{(i,j)\in\I}\overline{U}_{i,j}^2=\emptyset$. Then there exist an index $(i,j)\in\I$ and a subarc $\hat{\gamma}\subset A_{i,j}$ of $\gamma$ connecting a point $p_0\in A_{i,j}\setminus \Omega_{i,j}$ to a point $p_1\in \beta_{i,j}^1\subset A_{i,j}$; see Fig.\ \ref{fig:R}. In view of {\rm (f4)} and {\rm (f3)} we have that 
\begin{equation}
\label{eq:Fp0}
\text{$\wh F(p_0)\approx F_0(p_0)$\quad and\quad $\pi_{i,j}(\wh F(p_1))\approx \pi_{i,j}(F_0(p_1))$.}
\end{equation}
On the other hand, \eqref{eq:Aij}  gives that $\{F_0(p_0),F_0(p_1)\}\subset\Uscr_{i,j}$, and {\rm (f2)} that $\wh F(p_1)\in \b(\hat{s})\setminus \overline{\b}(\hat{s}-\hat{\epsilon})$. Therefore, \eqref{eq:Pytha} and \eqref{eq:Fp0} ensure that $\delta<\dist(\wh F(p_0),\wh F(p_1))\leq \length(\wh F(\hat{\gamma})) \leq \length(\wh F(\gamma))$ and we are done.
\end{proof}


\section{Proper null curves in $\c^3$ with a bounded coordinate function}
\label{sec:boundedNC}

This section is devoted to the proof of Theorem \ref{th:boundedcoordinate} in the following more precise form.

\begin{theorem}\label{th:main2}
Let $ M$ be a bordered Riemann surface. Every null holomorphic immersion $F=(F_1,F_2,F_3)\in\IA( M)$ can be approximated, uniformly on compacta in $ M$, by proper 
null holomorphic embeddings  $G=(G_1,G_2,G_3)\colon  M\to \c^3$ such that $|G_3|<\eta$ on $M$ for any given constant $\eta>\|F_3\|_{0,\overline  M}$. 
\end{theorem}

The proof relies on the approximation result furnished by Lemma \ref{lem:main2} below. 

We fix the following notation in the remainder of this section. Let $\{V_1,V_2\}$ be an orthonormal basis of $\c^2\times\{0\}\subset\c^3$ consisting of null vectors.
(We can take for instance $V_1=\frac{1}{\sqrt2} (1,\imath,0)$ and $V_2=\frac{1}{\sqrt2}(1,-\imath,0)$.) 
Given a set $K$ and a map $f\colon K\to\c^3$, set 
\begin{equation}
\label{eq:m(f)}
  \mgot(f)=\max\{|\langle f,V_1\rangle|,|\langle f,V_2\rangle|\}\colon K\to\r_+.
\end{equation}

%
%
%
%
\begin{lemma}\label{lem:main2}
Let $M$ be a bordered Riemann surface, let $\Rcal\in\Bscr( M)$, let $F=(F_1,F_2,F_3)\in \IA( M)$, let $s>0$ be a positive number, and assume that
\begin{equation}
\label{eq:lemma2}
\mgot(F)>s\quad \text{on $\overline{ M}\setminus  \Rcal$}. 
\end{equation}
Then, for any $\epsilon>0$ and $\hat s>s$, there exists a null immersion $\wh F=(\wh F_1,\wh F_2,\wh F_3) \in \IA( M)$ 
enjoying the following properties: 
\begin{enumerate}[\rm (L1)]
\item $\|\wh F -F\|_{1,\overline{\Rcal}}<\epsilon$.
\item $\mgot(\wh F)>\hat s$ on $bM$.
\item $\mgot(\wh F)>s$ on $\overline{ M}\setminus  \Rcal$.
\item $\|\wh F_3\|_{0,\overline M}<\|F_3\|_{0,\overline M}+\epsilon$.
\end{enumerate}
\end{lemma}

The main novelty of Lemma \ref{lem:main2} with respect to previous related constructions of minimal surfaces in $\r^3$ and null curves in $\c^3$ is the control on the third coordinate function provided by property {\rm (L4)}; cf.\ \cite[Lemma 5.1]{AL1} and \cite[Lemma 8.2]{AF2}.


\begin{proof}[Proof of Lemma \ref{lem:main2}]
Assume without loss of generality that $M$ is a bordered domain in an open Riemann surface $\widehat{ M}$ such that $M \in \Bscr(\widehat{ M})$ and, by the Mergelyan theorem for null curves \cite{AL2,AF2}, that $F\in\IA( M)$ extends to a null holomorphic immersion $F\in\IA(\wh  M)$.

The proof is accomplished in three different steps that follow the spirit of those in the proof of Lemma \ref{lem:main}. Again, the main novelties with respect to previous related constructions arise in the last step where the solutions to the Riemann-Hilbert problem for null curves (Theorem \ref{th:RH}) are used.

\noindent{\em Step 1: Splitting $bM$.} Denote by $ C_1,\ldots, C_\igot$ the connected boundary curves of $\overline{ M}$, so $bM=\bigcup_{i=1}^\igot  C_i$. For every $i = 1,\ldots,\igot$ choose a small annular neighborhood $A_i\subset\overline M$ of $ C_i$ and a smooth retraction $\rho_i\colon A_i\to C_i$ such that $\overline\Rcal \subset K:= M\setminus \bigcup_{i=1}^\igot A_i$.

In view of \eqref{eq:lemma2}, we can split $bM$ into a finite family of compact arcs $C_{i,j}$, $(i,j)\in\I:=\{1,\ldots,\igot\}\times\z_\jgot$, $\jgot\in\n$, so that the following properties hold for every $i\in\{1,\ldots,\igot\}$:
\begin{enumerate}[\rm ({a}1)]
\item $\bigcup_{j=1}^\jgot C_{i,j}=C_i$.
\item $C_{i,j}$ and $C_{i,j+1}$ have a common endpoint, $p_{i,j}$, and are otherwise disjoint.
\item Either $|\langle F,V_1\rangle|>s$ on $C_{i,j}$ or $|\langle F,V_2\rangle|>s$ on $C_{i,j}$ for every $j \in \z_\jgot$.
\end{enumerate}

Set $\I_1=\{(i,j)\in\I\colon |\langle F,V_1\rangle|>s$ on $C_{i,j}\}$ and $\I_2=\I\setminus \I_1$, so we have $|\langle F,V_2\rangle|>s$ on $C_{i,j}$ for all $(i,j)\in\I_2$ in view of {\rm (a3)}.

\noindent{\em Step 2: Stretching from the points $p_{i,j}$}. 
For each $(i,j)\in\I$ choose a compact smooth embedded arc $\lambda_{i,j}\subset\c^3$ enjoying the following conditions:
\begin{enumerate}[\rm ({b}1)]
\item $\lambda_{i,j}$ is attached to $F(\overline M)$ at $F(p_{i,j})$ and it is otherwise disjoint from $F(\overline M)$.
\item $|\langle z,V_k\rangle|>s$ for all $z\in \lambda_{i,j-1}\cup\lambda_{i,j}$ for all $(i,j)\in \I_k$.
\item $|\langle v_{i,l},V_k\rangle|>\hat s$ for all $l\in\{j-1,j\}$ and all $(i,j)\in \I_k$, where $v_{i,l}$ denotes the endpoint of $\lambda_{i,l}$ different from $F(p_{i,l})$.
\item $|\langle z, (0,0,1)\rangle|<\|F_3\|_{0,\overline M}+\epsilon/2$ for all $z\in\lambda_{i,j}$.
\end{enumerate}
Observe that {\rm (b1)} and {\rm (b2)} are compatible thanks to {\rm (a3)}. For {\rm (b4)} take into account that the third component of $V_1$ and $V_2$ equals $0$. 

Up to slightly modifying the arcs $\lambda_{i,j}$, we can use the method of exposing boundary points (see \cite[Theorem 2.3]{FW0} and also \cite[Theorem 8.8.1]{F2011}) in order to approximate $F$ in the $\Cscr^1$ topology outside a small open neighborhood of $\{p_{i,j}\colon (i,j)\in\I\}$ by a null curve $F^0\in\IA(M)$ containing these arcs and mapping the point $p_{i,j}$ to $v_{i,j}$ for all $(i,j)\in \I$. Furthermore, $F^0$ can be taken such that $F^0(\overline M)$ is arbitrarily close to $F(\overline M)\cup \bigcup_{(i,j)\in\I}\lambda_{i,j}$. More precisely, arguing as in Step 2 in the proof of Lemma \ref{lem:main}, we find a null curve $F^0=(F^0_1,F^0_2,F^0_3)\in\IA(M)$ enjoying the following properties:
\begin{enumerate}[\rm ({c}1)]
\item $\|F^0-F\|_{1,K}<\epsilon/2$.
\item $|\langle F^0,V_k\rangle|>s$ on $C_{i,j}$ for all $(i,j)\in\I_k$; see {\rm (b2)}.
\item $|\langle F^0,V_k\rangle|>\hat s$ on $\{p_{i,j-1},p_{i,j}\}$ for all $(i,j)\in \I_k$; see {\rm (b3)}.
\item $\|F^0_3\|_{0,\overline M}<\|F_3\|_{0,\overline M}+\epsilon/2$; see {\rm (b4)}.
\item $\mgot(F^0)>s$ on $\overline M\setminus \Rcal$; see {\rm (b2)} and \eqref{eq:lemma2}.
\end{enumerate}

\noindent{\em Step 3: Stretching from the arcs $C_{i,j}$.}
In view of {\rm (c3)}, the continuity of $F^0$ gives compact subarcs $\alpha_{i,j}$ of $C_{i,j}\setminus \{p_{i,j-1},p_{i,j}\}$, $(i,j)\in \I$, such that
\begin{equation}\label{eq:c3}
|\langle F^0,V_k\rangle|>\hat s\quad \text{on $\overline{C_{i,j}\setminus \alpha_{i,j}}$ for all $(i,j)\in\I_k$, $k=1,2$.}
\end{equation}
(See Fig.\ \ref{fig:lemma52}.)
\begin{figure}[ht]
    \begin{center}
    \resizebox{0.85 \textwidth}{!}{\includegraphics{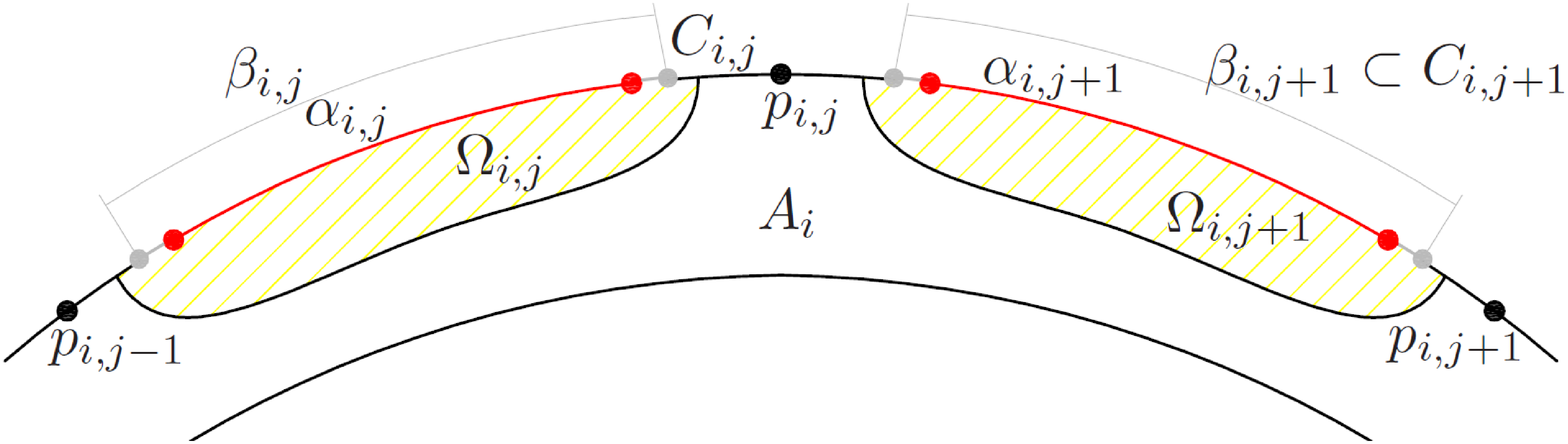}}
    \end{center}
        \vspace{-0.25cm}
\caption{The sets in $\overline M\setminus \overline{\Rcal}$.}\label{fig:lemma52}
\end{figure}
From \eqref{eq:c3}, we can easily find compact subarcs $\beta_{i,j}$ of $C_{i,j}\setminus \{p_{i,j-1},p_{i,j}\}$, containing $\alpha_{i,j}$ in its relative interior, $(i,j)\in \I$, and a continuous function $\mu\colon bM\to\r_+$ such that 
\begin{equation}\label{eq:mu1}
\mu=0\quad \text{on $bM\setminus \bigcup_{(i,j)\in\I} \beta_{i,j}$},
\end{equation}
and the map $\varkappa\colon bM\times\overline\d\to\c^3$ given by
\begin{equation}\label{eq:varkappa1}
\varkappa(p,\xi)=\left\{
\begin{array}{ll} 
F^0(p); & p\in bM\setminus \bigcup_{(i,j)\in\I} \beta_{i,j} \\
F^0(p)+\mu(p)\xi V_k; & p\in \beta_{i,j},\; (i,j)\in\I\setminus\I_k,\; k=1,2,\\
\end{array}\right.
\end{equation}
enjoys the following properties:
\begin{enumerate}[\rm ({d}1)]
\item $|\langle \varkappa(p,\xi), V_k\rangle|>\hat s$ for all $p\in \alpha_{i,j}$ and $\xi\in b\d$, $(i,j)\in\I\setminus\I_k$, $k=1,2$.
\item $|\langle \varkappa(p,\xi), V_k\rangle|>\hat s$ for all $p\in \overline{C_{i,j}\setminus\alpha_{i,j}}$ and $\xi\in \overline\d$, $(i,j)\in\I_k$, $k=1,2$; see \eqref{eq:c3}, \eqref{eq:varkappa1}, and recall that $\langle V_1,V_2\rangle=0$.
\item $|\langle \varkappa(p,\xi), V_k\rangle|>s$ for all $p\in C_{i,j}$ and $\xi\in \overline\d$, $(i,j)\in\I_k$, $k=1,2$; see {\rm (c2)}, \eqref{eq:varkappa1}, and take into account that $\langle V_1,V_2\rangle=0$.
\end{enumerate}
Notice that \eqref{eq:mu1} and the continuity of $\mu$ imply the one of $\varkappa$; observe that the arcs $\beta_{i,j}$, $(i,j)\in \I$, are pairwise disjoint.

We can now use Theorem \ref{th:RH} to obtain small open neighborhoods $\Omega_{i,j}\subset A_i\subset\overline M\setminus K$ of the arcs $\beta_{i,j}$, $(i,j)\in\I$, and a holomorphic null immersion $\wh F=(\wh F_1,\wh F_2,\wh F_3)\in\IA(M)$ satisfying the following conditions (see Fig.\ \ref{fig:lemma52}):
\begin{enumerate}[\rm ({e}1)]
\item $\overline\Omega_{i,j}\cap \overline \Omega_{i,l}=\emptyset$ for all $l\neq j$, and $\overline\Omega_{i,j}\cap bM$ is a compact arc in $C_{i,j}\setminus\{p_{i,j-1},p_{i,j}\}$ containing $\beta_{i,j}$ in its relative interior.
\item $|\langle \wh F, V_k\rangle|>\hat s$ on $\alpha_{i,j}$, $(i,j)\in\I\setminus\I_k$, $k=1,2$; see {\rm (d1)} and Theorem \ref{th:RH}-{\it i)}.
\item $|\langle \wh F, V_k\rangle|>\hat s$ on $\overline{C_{i,j}\setminus\alpha_{i,j}}$, $(i,j)\in\I_k$, $k=1,2$; see {\rm (d2)} and Theorem \ref{th:RH}-{\it ii)}.
\item $\|\wh F-F^0\|_{1,\overline M\setminus \bigcup_{(i,j)\in \I}\Omega_{i,j}} <\epsilon/2$; see Theorem \ref{th:RH}-{\it iii)}.
\item $\|\wh F_3\|_{0,\overline M}<\|F^0_3\|_{0,\overline M}+\epsilon/2$; see \eqref{eq:varkappa1}, Theorem \ref{th:RH}-{\it ii)}, {\it iii)},  and recall that $V_1$ and $V_2$ are orthogonal to $(0,0,1)$.
\item $\mgot(\wh F)>s$ on $\overline M\setminus\Rcal$; see {\rm (d2)}, {\rm (d3)}, {\rm (c5)}, Theorem \ref{th:RH}-{\it i)}, {\it ii)}, and {\it iii)}.
\end{enumerate} 

The immersion $\wh F\in\IA(M)$ meets all the requirements of the lemma. Indeed, {\rm (L1)} is implied by {\rm (c1)} and {\rm (e4)}; {\rm (L2)} is ensured by {\rm (e2)} and {\rm (e3)}; {\rm (L3)} agrees with {\rm (e6)}; and {\rm (L4)} is given by {\rm (c4)} and {\rm (e5)}. This concludes the proof of Lemma \ref{lem:main2}.
\end{proof}

\begin{remark}
A reader familiar with the standard constructions of proper holomorphic maps (see for instance the
introduction and references in \cite{DF2007}, or \cite[\S 8.6]{F2011}) might ask why are we using the additional device of exposing boundary points in the proof of Lemma \ref{lem:main2}. 
Indeed, if one could prove Theorem \ref{th:RH} with the size function
$\mu\colon bM\to \r_+$ supported on the entire boundary of $M$ (and not only on arcs), then such a proof would be possible by following the zig-zag procedure outlined in the introduction for the case when $M$ is the disc.
However, we find this a minor issue, in particular since the exposing of points technique 
is an integral part of our analysis already in \S\ref{sec:th} above.
\end{remark}

\begin{proof}[Proof of Theorem \ref{th:main2}] 
Let $K\subset M$ be a compact subset. Pick a constant $\varepsilon>0$ such that 
\begin{equation}
\label{eq:F3eta}
\|F_3\|_{0,\overline M}+\varepsilon<\eta.
\end{equation}

Set $F^0=(F^0_1,F^0_2,F^0_3):=F$ and $\varepsilon_0:=\varepsilon$. By the general position argument we may assume that $\mgot(F^0)=\max\{|\langle F^0,V_1\rangle|\,,\, |\langle F^0,V_2\rangle|\}$ does not vanish anywhere on $b  M$; hence there exists a constant $s_0>0$ and a bordered domain $\Rcal_0\in\Bscr( M)$ containing $K$  such that 
\begin{equation}
\label{eq:n=0}
\mgot(F^0)>s_0\quad \text{on $\overline M\setminus\Rcal_0$}.
\end{equation}
Set $s_n:=s_0+n$ for $n\in\n$. 
We shall inductively construct sequences of bordered domains $\Rcal_n\in\Bscr( M)$, null holo\-morphic embeddings $F^n=(F^n_1,F^n_2,F^n_3)\in\IA( M)$, and constants $\varepsilon_n>0$, $n\in\n$, satisfying the following properties:
\begin{enumerate}[\rm (a$_n$)]
\item $\Rcal_n\Supset \Rcal_{n-1}$.
\item $\displaystyle 0<\varepsilon_n <\varepsilon_{n-1}/2$ and every holomorphic map $G\colon \overline  M \to\c^3$ satisfying $\|G-F^{n-1}\|_{1,\overline{\Rcal}_{n-1}}<2\varepsilon_n$ is an embedding on $\overline{\Rcal}_{n-1}$.
\item $\|F^n-F^{n-1}\|_{1,\overline{\Rcal}_{n-1}}<\varepsilon_n$.
\item $\mgot(F^n)>s_n$ on $\overline M\setminus\Rcal_n$.
\item $\mgot(F^n)>s_{n-1}$ on $\overline M\setminus\Rcal_{n-1}$.
\item $\|F^n_3\|_{0,\overline M}<\|F^{n-1}_3\|_{0,\overline M}+\varepsilon_n$.
\end{enumerate}
We will also ensure that 
\begin{equation}\label{eq:cupM}
	 M=\bigcup_{n\in\n}\Rcal_n.
\end{equation}

For the basis of the induction, take a number $0<\varepsilon_1<\varepsilon/2$ small enough so that {\rm (b$_1$)} holds; recall that $F^0\in\IA( M)$ is an embedding. Condition \eqref{eq:n=0} shows that we can apply Lemma \ref{lem:main2} to the data
\[
(M \,,\, \Rcal \,,\, F \,,\, s  \,,\, \epsilon \,,\, \hat s)= (M \,,\, \Rcal_0 \,,\, F^0 \,,\, s_0  \,,\, \varepsilon_1 \,,\, s_1),
\]
obtaining a null holomorphic immersion $F^1\in\IA( M)$, which can be assumed to be an embedding by general position argument \cite[Theorem 2.5]{AF2}. Conditions {\rm (c$_1$)}, {\rm (e$_1$)}, and {\rm (f$_1$)} trivially follow, whereas {\rm (a$_1$)} and {\rm (d$_1$)} are ensured by Lemma \ref{lem:main2}-{\rm (L2)} provided that the bordered domain $\Rcal_1\in\Bscr( M)$ is chosen large enough.

For the inductive step, assume that for some $n>1$ we already have a triple $(\Rcal_{j},F^{j},\varepsilon_{j})$ enjoying properties {\rm (a$_j$)}--{\rm (f$_j$)} for all $j\in\{1,\ldots,n-1\}$. Take any $\varepsilon_n>0$ satisfying {\rm (b$_n$)}; again the existence of such a number is ensured by the fact that $F^{n-1}\in\IA( M)$ is an embedding. In view of {\rm (d$_{n-1}$)} we can apply Lemma \ref{lem:main2} to the data
\[
	(M \,,\, \Rcal \,,\, F \,,\, s  \,,\, \epsilon \,,\, \hat s)= (M \,,\, \Rcal_{n-1} \,,\, F^{n-1} \,,\, s_{n-1}  \,,\, \varepsilon_n \,,\, s_n),
\]
obtaining $F^n\in\IA( M)$ which can be assumed an embedding \cite{AF2} and which, together with $\varepsilon_n$ and a sufficiently large domain $\Rcal_n \in \Bscr( M)$, meets all requirements {\rm (a$_n$)}--{\rm (f$_n$)}. 

Finally, in order to ensure \eqref{eq:cupM}, we simply choose the domain $\Rcal_n$ large enough in each step of the inductive process. This concludes the construction of the sequence $\{(\Rcal_n,F^n,\varepsilon_n)\}_{n\in\n}$.

Since $F^n\in\IA( M)$ for all $n\in\n$, properties {\rm (c$_n$)}, {\rm (b$_n$)}, and (\ref{eq:cupM}) show  that the sequence $\{F^n\in\IA( M)\}_{n\in\n}$ converges uniformly on compacta in $ M$ to a holomorphic map $G =(G_1,G_2,G_3)\colon  M\to\c^3$ whose derivative, $G'$, with respect to any local holomorphic coordinate on $ M$ assumes values in $\Agot$. 

Let us show that $G$ satisfies Theorem \ref{th:main2}. From {\rm (b$_n$)} and {\rm (c$_n$)} it follows that
\begin{equation}\label{eq:hatFclose2}
\|G-F^j\|_{1,\overline \Rcal_j}<2\varepsilon_{j+1}<\varepsilon_j\leq \varepsilon/2^j\quad \forall j\in\n\cup\{0\}.
\end{equation}
This and {\rm (c$_n$)} imply that $G\colon M\to\c^3$ is a null holomorphic embedding which is $\varepsilon$-close to $F^0=F$ in the $\Cscr^1$ topology on the domain $\Rcal_0\supset K$. 

Next we check that $G\colon M\to\c^3$ is a proper map. Indeed, by {\rm (e$_n$)} and \eqref{eq:hatFclose2}, we have
\[
\mgot(G)> s_{n-1}-\varepsilon/2^n\quad \text{on  $\overline \Rcal_n\setminus \Rcal_{n-1}$,  for all $n\in\n$}
\]
(see \eqref{eq:m(f)}). Together with \eqref{eq:cupM} and the fact that $\{s_n\}_{n\in\n}\to+\infty$, it follows that the map  $\mgot(G)\colon  M\to\r$ is proper; hence so is $G$; cf.\ \eqref{eq:m(f)}.

Finally, {\rm (b$_n$)}, {\rm (f$_n$)}, and \eqref{eq:F3eta} ensure that $|G_3|<\eta$ on $M$, thus concluding the proof.
\end{proof}


Lemma \ref{lem:main2} can be also combined with the Mergelyan theorem for null curves in $\c^3$ \cite{AF2,AL1} in order to prove the first assertion in Theorem \ref{th:giventopology}. Afterward, the second and third assertions there follow as explained in the introduction. We can actually prove the following more precise result in the line of the first part of Theorem \ref{th:giventopology}.

\begin{theorem}\label{th:topology}
Assume that $N$ is an orientable noncompact smooth real surface without boundary, $K$ is a compact subset of $N$ such that $N\setminus K$ does not have any relatively compact connected components in $N$, and $J_K$ is a complex structure on an open neighborhood $U\subset N$ of $K$.
Let $F=(F_1,F_2,F_3)\in\IA(U)$ be a null holomorphic immersion. Then for any $\epsilon>0$ there exist a complex structure $J$ on $N$ and a proper null holomorphic embedding $G=(G_1,G_2,G_3)$ of the Riemann surface $\Ncal:=(N,J)$ to $\c^3$ such that:
\begin{enumerate}[\it i)]
\item $J=J_K$ on a neighborhood of $K$, 
\item $\|G-F\|_{1,K}<\epsilon$, and 
\item $|G_3|<\|F_3\|_{0,K}+\epsilon$ on $N$.
\end{enumerate}
\end{theorem} 

\begin{proof} 
If $N$ has finite topological type then the conclusion is implied by Theorem \ref{th:main2}. 

Assume now that $N$ has infinite topology. Choose a bordered Riemann surface $M_0$, with $K\subset M_0 \Subset U$, such that $N\setminus M_0$ has no relatively compact connected components. Choose a complex structure $J_0$ on $N$ such that 
\begin{equation}
\label{eq:J0}
		J_0=J_K \quad\text{on} \ \overline M_0
\end{equation}
and denote by $\Ncal_0:=(N,J_0)$ the corresponding Riemann surface.  In the sequel we use $J_0$ as the complex structure on subdomains of $N$.

Note that $\overline M_0$ is Runge in $\Ncal_0$. Pick an exhaustion $\{\overline M_j\}_{j\in\n}$ of $\Ncal_0$ by smoothly bounded compact domains 
such that $\overline M_{j-1}\subset M_j$, every $M_j$ is Runge in $\Ncal_0$, and the difference $\overline M_j\setminus M_{j-1}$ has the Euler characteristic $\chi(\overline M_j\setminus M_{j-1})=-1$ for every $j\in\n$; see \cite[Lemma 4.2]{AL3}.

Set $\Rcal_0:=M_0$. Since $\overline\Rcal_0$ is Runge in $M_1$, the inclusion map $\overline\Rcal_0\hookrightarrow \overline M_1$ induces a monomorphism of the homology groups $H_1(\overline\Rcal_0;\z)\hookrightarrow H_1(\overline M_1;\z)$, hence we may assume that $H_1(\overline \Rcal_0;\z)\subsetneq H_1(\overline M_1;\z)$. (Recall that $\chi(\overline M_1\setminus\Rcal_0)=-1$, so the inclusion is proper.) There is a smooth Jordan curve $\hat\gamma \subset M_1$ intersecting $M_1\setminus \Rcal_0$ in a compact connected Jordan arc $\gamma$ with endpoints $a, b\in b\Rcal_0$ and otherwise disjoint from $\overline\Rcal_0$ such that  
\begin{equation}
\label{eq:homo}
H_1(\overline\Rcal_0\cup\gamma;\z)=H_1(\overline M_1;\z).
\end{equation}

Set $F^0=(F^0_1,F^0_2,F^0_3):=F$ and $\epsilon_0:=\epsilon$. By a general position argument, we may assume that $\mgot(F^0)>s_0>0$ on $b M_0$ for some constant $s_0>0$; see \eqref{eq:m(f)}. Set $s_n:=s_0+n$ for all $n\in\n$.

We may assume without loss of generality that $F^0\colon \overline\Rcal_0\to\c^3$ is a null embedding \cite[Theorem 2.5]{AF2}. Hence there is a number $\epsilon_1$ with $0<\epsilon_1<\epsilon_0/2$ such that every holomorphic map $Y\colon \overline\Rcal_0\to\c^3$ satisfying $\|Y-F^0\|_{1,\overline\Rcal_0}< 2\epsilon_1$ is an embedding.

We extend $F^0$, with the same name, to a smooth injective map $\overline\Rcal_0\cup\gamma\to\c^3$ so that the following conditions are satisfied:
\begin{itemize}
\item $\mgot(F^0)>s_0$ on $b\Rcal_0\cup\gamma$.
\item $\|F^0_3\|_{0,\overline\Rcal_0\cup\gamma}<\|F^0_3\|_{0,\overline\Rcal_0}+\epsilon_1/2$.
\item The tangent vector to $F^0$ at any point of $\gamma$ lies in $\Agot\setminus\{0\}$.
\end{itemize}

Observe that the compact set $\overline\Rcal_0\cup\gamma\subset \Ncal_0$ is {\em admissible} in the sense of \cite[Def.\ 2.2]{AL1} (see also \cite[Def.\ 2]{AL2} or \cite[Def.\ 7.1]{AF2}), and that $F^0\colon \overline\Rcal_0\cup\gamma\to\c^3$ is a {\em generalized null curve} in the sense of \cite[Def. 5]{AL2} or \cite{AF2}. Therefore, the Mergelyan theorem for null curves in $\c^3$ \cite{AF2,AL1,AL2} applies, providing a smoothly bounded compact domain $\overline\Rcal_1 \subset M_1$ and a null curve $\wh F^1=(\wh F^1_1,\wh F^1_2,\wh F^1_3)\in\IA(\Rcal_1)$ enjoying the following properties:
\begin{itemize}
\item $\overline\Rcal_0\cup\gamma\subset\Rcal_1\Subset M_1$ and the inclusion map $\overline\Rcal_1\hookrightarrow \overline M_1$ induces an isomorphism between the homology groups $H_1(\overline\Rcal_1;\z)\stackrel{\cong}{\longrightarrow} H_1(\overline M_1;\z)$; see \eqref{eq:homo}.
\item $\|\wh F^1-F^0\|_{1,\overline\Rcal_0}<\epsilon_1/2$.
\item $\mgot(\wh F^1)>s_0$ on $\overline\Rcal_1\setminus\Rcal_0$.
\item $\|\wh F^1_3\|_{0,\overline\Rcal_1}<\|F^0_3\|_{0,\overline\Rcal_0}+\epsilon_1/2$.
\end{itemize}
Then, we can apply Lemma \ref{lem:main2} to the data
\[
(M \,,\, \Rcal \,,\, F \,,\, s  \,,\, \epsilon \,,\, \hat s)= (M_1 \,,\, \Rcal_1 \,,\, \wh F^1 \,,\, s_0  \,,\, \epsilon_1/2 \,,\, s_1),
\]
obtaining $F^1\in\IA(\Rcal_1)$, which can be assumed an embedding \cite{AF2}, such that:
\begin{itemize}
\item $\|F^1-F^0\|_{1,\overline\Rcal_0}<\epsilon_1$.
\item $\mgot(F^1)>s_0$ on $\overline\Rcal_1\setminus\Rcal_0$.
\item $\mgot(F^1)>s_1$ on $b\Rcal_1$.
\item $\|F^1_3\|_{0,\overline\Rcal_1}<\|F^0_3\|_{0,\overline\Rcal_0}+\epsilon_1$.
\end{itemize}

By repeating the above procedure inductively as in the proof of Theorem \ref{th:main2}, we find sequences of bordered domains $\Rcal_n\Subset \Ncal_0$, null holomorphic embeddings $F^n\colon\overline\Rcal_n\to\c^3$, and numbers $\epsilon_n>0$, $n\in\n$, such that the 
following conditions hold for every $n\in \n$:
\begin{enumerate}[\rm (a)]
\item $\Rcal_{n-1}\Subset\Rcal_n\subset M_n$ and the inclusion map $\overline\Rcal_n\hookrightarrow \overline M_n$ induces an isomorphism between the homology groups $H_1(\overline\Rcal_n;\z)\stackrel{\cong}{\longrightarrow} H_1(\overline M_n;\z)$.
\item $\epsilon_n<\epsilon_{n-1}/2\leq \epsilon/2^n$ and every holomorphic map $Y\colon \overline\Rcal_n\to\c^3$ satisfying $\|Y-F^{n-1}\|_{1,\overline\Rcal_{n-1}}<2\epsilon_n$ is an embedding on $\overline\Rcal_{n-1}$.
\item $\|F^n-F^{n-1}\|_{1,\overline\Rcal_{n-1}}<\epsilon_n$.
\item $\mgot(F^n)>s_n$ on $b\Rcal_n$.
\item $\mgot(F^n)>s_{n-1}$ on $\overline\Rcal_n\setminus\Rcal_{n-1}$.
\item $\|F^n_3\|_{0,\overline\Rcal_n}<\|F^{n-1}_3\|_{0,\overline\Rcal_{n-1}}+\epsilon_n/2$.
\end{enumerate}

The union $N_0:=\bigcup_{n\in\n}\Rcal_n\subset N$ is an open domain in $N$.  Since $N=\bigcup_{n\in\n} M_n$, it follows from {\rm (a)} that the inclusion $N_0 \hookrightarrow N$ induces an isomorphism $H_1(N_0;\z)\stackrel{\cong}{\longrightarrow} H_1(N;\z)$. Therefore $N_0$ is homeomorphic to $N$, and hence also diffeomorphic to $N$. Fix a diffeomorphism $\phi\colon N \to N_0$ which equals the identity map on a neighborhood of $K$, and let $J=\phi^* (J_0|_{N_0})$ be the complex structure on $N$ obtained by pulling back the complex structure $J_0$ by $\phi$. The Riemann surface $\Ncal:=(N,J)$ is then biholomorphic to $(N_0,J_0|_{N_0})$ via $\phi$, and we have $J=J_0=J_K$ on a neighborhood of $K$ (see \eqref{eq:J0}).

Taking into account conditions {\rm (b)} and {\rm (c)}, the sequence $\{F^n\}_{n\in\n}$ converges uniformly on compacta in $N_0$ to a null holomorphic map $\widehat G\colon N_0\to\c^3$ (with respect to the complex structure $J_0$) satisfying
\begin{equation}
\label{eq:G-close}
	\|\widehat G-F^n\|_{1,\overline\Rcal_n}<2\epsilon_{n+1}<\epsilon_n\quad \forall n\in\n\cup\{0\}.
\end{equation}
Therefore $\widehat G$ is an embedding by {\rm (b)}, and property Theorem \ref{th:topology}-{\it ii)} holds.
Moreover, \eqref{eq:G-close} and {\rm (e)} show that $\mgot(\widehat G)\colon N_0\to\r$ is a proper map; hence so is $\widehat G$. Finally, properties {\rm (f)} and {\rm (b)} ensure {\it iii)}. The composed map $G:=\widehat G \circ \phi\colon N\to \c^3$ is then a null embedding of the Riemann surface $\Ncal=(N,J)$ satisfying Theorem \ref{th:topology}.
\end{proof}


\subsection*{Acknowledgements}
A.\ Alarc\'{o}n is supported by Vicerrectorado de Pol\'{i}tica Cient\'{i}fica e Investigaci\'{o}n de la Universidad de Granada, and is partially supported by MCYT-FEDER grants MTM2007-61775 and MTM2011-22547, Junta de Andaluc\'{i}a Grant P09-FQM-5088, and the grant PYR-2012-3 CEI BioTIC GENIL (CEB09-0010) of the MICINN CEI Program. 

F.\ Forstneri\v c is supported by the program P1-0291 and the grant J1-5432 from ARRS, Republic of Slovenia.


\vskip 0.5cm

\noindent Antonio Alarc\'{o}n

\noindent Departamento de Geometr\'{\i}a y Topolog\'{\i}a, Universidad de Granada, E-18071 Granada, Spain.

\noindent  e-mail: {\tt alarcon@ugr.es}

\vspace*{0.3cm}

\noindent Franc Forstneri\v c

\noindent Faculty of Mathematics and Physics, University of Ljubljana, and Institute
of Mathematics, Physics and Mechanics, Jadranska 19, 1000 Ljubljana, Slovenia.

\noindent e-mail: {\tt franc.forstneric@fmf.uni-lj.si}
\end{document}